\documentclass[11pt]{amsart}

\usepackage{mathrsfs}
\usepackage{amsthm,amsmath,amssymb,amsfonts}
\usepackage{enumerate}
\usepackage{tikz-cd}
\usepackage{adjustbox}
\usepackage{esint}
\usepackage[pagebackref,colorlinks=true,linkcolor=blue,citecolor=blue,urlcolor=blue]{hyperref}

\numberwithin{equation}{section}

\def\mc{\mathcal}
\def\mb{\mathbb}
\def\Bl{{\rm Bl}}
\def\Pic{{\rm Pic}}
\def\Reg{{\rm Reg}}

\def\Str{{\rm Str}}
\def\im{{\rm im}}
\def\id{{\rm id}}

\theoremstyle{plain}
\newtheorem{thm}{Theorem}[section]
\newtheorem{mthm}[thm]{Main Theorem}
\newtheorem{lemma}[thm]{Lemma}

\newtheorem{prop}[thm]{Proposition}
\newtheorem{cor}[thm]{Corollary}

\newtheorem{quest}[thm]{Question}
\newtheorem{Th}[thm]{Theorem}
\newtheorem{Le}[thm]{Lemma}
\newtheorem{Cor}[thm]{Corollary}
\newtheorem{Pro}[thm]{Proposition}

\theoremstyle{definition}
\newtheorem{defn}[thm]{Definition}
\newtheorem{rem}[thm]{Remark}

\newtheorem{De}[thm]{Definition}
\newtheorem{Rem}[thm]{Remark}
\newtheorem{Ex}[thm]{Example}

\newtheorem*{rem*}{Remark}

\allowdisplaybreaks

\title[Bimeromorphic invariance and blow-up formulae I]
{Bimeromorphic invariance of $\partial\bar\partial$-property and blow-up formulae I:
counterexamples and derived category approach}

\author{Wei Liu}
\address{School of Mathematics and statistics, Wuhan  University, Wuhan 430072, China}
\email{2024102010011@whu.edu.cn}

\author{Sheng Rao}
\address{School of Mathematics and statistics, Wuhan  University, Wuhan 430072, China}
\email{likeanyone@whu.edu.cn}
\date{\today}
\thanks{The authors are  partially supported by NSFC (Grant No. 12271412, W2441003, 12671107) and Hubei Provincial Innovation Research Group Project (Grant No. 2025AFA044).}


\begin{document}

\begin{abstract}
Using N.~Honda's construction of twistor spaces, bimeromorphic  invarince of $\partial\bar\partial$-lemma for threefolds and embedded resolution, we find a smooth compact complex threefold satisfying the
$\partial\bar\partial$-lemma that contains a smooth compact surface
for which the lemma fails.  This disproves heredity in the smallest
possible ambient dimension and yields counterexamples to
L. Alessandrini's modification question in every complex dimension at
least four.  We establish a derived blow-up formula for generalized
Bott--Chern complexes with coefficients, compatible with the comparison
maps to coefficient cohomology.
We also obtain a Dolbeault blow-up formula with bounded derived
coefficients.  For compact complex manifolds, we introduce a relative
$\partial\bar\partial$-property, characterize it using the
Fr\"olicher spectral sequence and Hodge filtrations, and prove a
blow-up criterion.  Finally, counterexamples show the limitations of
the classical formulae for coherent coefficients and singular ambient
spaces, and the failure of the naive Bott--Chern K\"unneth formula.
\end{abstract}

\maketitle

\setcounter{tocdepth}{1}
\tableofcontents

\section{Introduction}
\label{Introduction}

Let $X$ be a compact complex manifold.  The $\partial\bar\partial$-lemma is a
cohomological property weaker than the existence of a K\"ahler metric but
strong enough to imply the Hodge decomposition, degeneration of the
Fr\"olicher spectral sequence at the first page, and formality of the de Rham
algebra.  It is satisfied by compact K\"ahler manifolds and, more generally,
by compact complex manifolds in the class $\mc C$ of Fujiki.

The behavior of the $\partial\bar\partial$-lemma under modifications has been
studied from several cohomological viewpoints.  A.~N.~Parshin
\cite[\S\S~4 and 5]{Pr66} and P.~Deligne--P.~Griffiths--J.~Morgan--D.~Sullivan \cite[Theorem 5.22]{DGMS75} proved that the property descends
under a modification: if $\mu:\tilde{X}\to X$ is a modification of compact complex
manifolds and $\tilde{X}$ satisfies the $\partial\bar\partial$-lemma, then so does
$X$.  L.~Alessandrini posed the converse question.

\begin{quest}[{\cite[\S~1]{Al17}}]\label{quest-modification}
If $X$ satisfies the $\partial\bar\partial$-lemma, must $\tilde{X}$ also satisfy the
$\partial\bar\partial$-lemma?
\end{quest}

S.~Rao--S.~Yang--X.~Yang
\cite[Theorem 1.6]{RYY19}\cite[Theorem 1.2]{RYY20} first studied the blow-up
criterion through the Deligne--Griffiths--Morgan--Sullivan characterization
of the $\partial\bar\partial$-lemma.  S.~Yang--X.~Yang
\cite[Theorem 1.3]{YY20} treated the three-dimensional case using the
characterization of D.~Angella--A.~Tomassini.  Shortly thereafter,
Angella--T.~Suwa--N.~Tardini--Tomassini
\cite[Theorem 13]{ASTT20} obtained a related result using
\v{C}ech--Dolbeault cohomology under additional assumptions.  Eventually,
J.~Stelzig established a Bott--Chern blow-up
formula and stated the criterion explicitly \cite[Corollary 1.40]{Se18}; see also \cite[\S~5]{Se21}.
Subsequently, Rao--Y.~Zou \cite[Corollary 5.4]{RZ24} gave another proof
using Varouchas cohomology.  In particular, S.~Yang--X.~Yang, Stelzig and Rao--Zou proved
bimeromorphic invariance for compact complex threefolds by different
methods. In this context,  L.~Meng reduced Question \ref{quest-modification} to the following heredity
question in \cite[Question 1]{Mn21}.

\begin{quest}\label{quest-heredity}
Let $Y$ be a closed complex submanifold of a compact complex
$\partial\bar\partial$-manifold $X$.  Does $Y$ necessarily satisfy the
$\partial\bar\partial$-lemma?
\end{quest}

The answer is affirmative when $\dim_{\mb C}X\leq2$, since every compact
complex curve is K\"ahler and a compact complex surface satisfies the
$\partial\bar\partial$-lemma precisely when its first Betti number is even.
Thus complex dimension three is the first dimension in which a counterexample
to heredity can occur.  The first part of this paper finds such a
counterexample.

\begin{mthm}\label{main-thm}
There exist a smooth compact complex threefold $\widehat Z$ and a smooth
compact connected divisor $D\subset\widehat Z$ such that
\begin{enumerate}[(i)]
\item $\widehat Z$ satisfies the $\partial\bar\partial$-lemma;
\item $D$ is a closed embedded complex submanifold of $\widehat Z$;
\item $D$ does not satisfy the $\partial\bar\partial$-lemma.
\end{enumerate}
More precisely, $D$ is a divisor bimeromorphic to a Hopf surface and
$b_1(D)=1$.
\end{mthm}

This construction also yields many examples of compact complex
$\partial\bar\partial$-manifolds outside Fujiki's class $\mathcal C$.
Moreover, it settles the modification problem in higher dimensions.

\begin{cor}\label{cor-alessandrini-negative}
For every integer $n\geq4$, there exists a modification
\[
\pi:\tilde{X}^{n}\longrightarrow X^{n}
\]
between smooth compact complex $n$-folds such that $X^{n}$ satisfies the
$\partial\bar\partial$-lemma whereas $\tilde{X}^{n}$ does not.  Hence
Question~\ref{quest-modification} has a negative answer in every complex
dimension at least four.
\end{cor}

Indeed, take
\[
X^{n}=\widehat Z\times\mathbb P^{n-3}
\quad\text{and}\quad
Z^{n-2}=D\times\mathbb P^{n-4},
\]
where $\mathbb P^{n-4}\subset\mathbb P^{n-3}$ is a linear subspace.  The
$\partial\bar\partial$-lemma is stable under products, so $X^{n}$ satisfies
it.  On the other hand, $Z^{n-2}$ has first Betti number one and therefore does
not satisfy the lemma.  Since $Z^{n-2}$ has codimension two in $X^{n}$, the blow-up
criterion shows that $\tilde{X}^{n}=\Bl_{Z^{n-2}}X^{n}$ does not satisfy the
$\partial\bar\partial$-lemma.  A complete proof is given after the proof of
Theorem~\ref{main-thm}.

The geometric input comes from N.~Honda's construction of twistor spaces of
algebraic dimension one on $n\mathbb P^2$, $n\geq5$.  Such a twistor threefold $Z$
contains an irreducible non-normal divisor $D_0$ bimeromorphic to a Hopf
surface \cite[Proposition 3.7.(ii) and Theorem 3.11.(ii)]{Hn15}.  Since $D_0$
is not smooth, it cannot directly serve as the submanifold in
Theorem~\ref{main-thm}.  We therefore take an embedded resolution
\[
 \mu:(\widehat Z,D)\longrightarrow(Z,D_0).
\]
This simultaneously resolves $D_0$ and embeds its smooth strict transform
$D$ in the smooth modification $\widehat Z$ of $Z$.

Two facts control the ambient manifold.  First, Honda's examples arise as
arbitrarily small deformations of Moishezon twistor spaces associated with
self-dual metrics constructed by D.~Joyce.  Openness of the
$\partial\bar\partial$-lemma under small deformations allows $Z$ to be chosen
inside the $\partial\bar\partial$ locus.  Second, every nontrivial center of
the embedded resolution is a point or a smooth compact curve.  Such centers
are K\"ahler, so the blow-up criterion preserves the
$\partial\bar\partial$-lemma at every step.  On the other hand, $D$ is
bimeromorphic to a Hopf surface and therefore has $b_1(D)=1$.  The
degree-one Hodge decomposition forced by the $\partial\bar\partial$-lemma
would make $b_1(D)$ even, proving that $D$ does not satisfy the lemma.

The second part of the paper concerns the derived structure of blow-up
formulae.  The cohomological decompositions have been studied in several
settings.  Rao--S.~Yang--X.~Yang \cite[Theorem 1.2]{RYY20} proved the
bundle-valued Dolbeault blow-up formula for compact complex manifolds,
and Meng \cite[Theorem 6.6]{Mn19MV} obtained its non-compact extension
using Mayer--Vietoris systems.  Stelzig \cite[\S~5]{Se21} described the double
complex of a blow-up up to $E_1$-isomorphism, thereby treating several
cohomology theories simultaneously.  Meng \cite[\S\S~4 and~5]{Mn20BC} constructed
blow-up maps for twisted cohomologies with supports at the level of
complexes of forms and currents; in particular, 
\cite[Proposition 4.15]{Mn20BC} gives the ordinary complex Bott--Chern
formula without compactness.  For Bott--Chern hypercohomology,
S.~Yang--X.~Yang \cite[Theorem 3.7]{SY} established a blow-up formula with an
explicit morphism by a sheaf-theoretic approach.  Integral Bott--Chern
blow-up formulae were obtained by Y.~Chen--S.~Yang \cite[Theorem 1.2]{CSY} and
X.~Wu \cite[Proposition 12]{Wu} in the compact setting.

Our aim is not merely to obtain another decomposition of cohomology
groups, or to remove compactness from an already known formula.  We
construct compatible decompositions of the coefficient complex and the
holomorphic and antiholomorphic truncated de Rham complexes, and pass
from these decompositions to their mapping cones.  This gives a common
derived formulation for different coefficient systems and retains the
comparison maps to coefficient cohomology.  These compatibilities are
important: an abstract isomorphism of cohomology groups alone does not
determine the behavior of the kernels of comparison maps, which measure
the failure of the relative $\partial\bar\partial$-property introduced
below.

Let $S$ be a commutative ring with unit, and
$f,g:S\longrightarrow\mathbb C$ ring morphisms.  For integers $p,q$, set
\[
\mathcal B_X^{p,q}(f,g):=
\operatorname{Cone}\!\left(S_X\xrightarrow{\Delta_{f,g}}
\Omega_X^{[0,p-1]}\oplus\overline{\Omega_X^{[0,q-1]}}\right)[-1],
\]
where $S_X$ is the constant sheaf with values in $S$ over a complex manifold $X$, placed in degree zero,
$\Omega_X^{[p,q]}$ denotes the truncated holomorphic de Rham complex
with $\Omega_X^k$ in degree $k$, and $\Delta_{f,g}$ is induced by
$(f,g)$ in degree zero.  Empty truncations are zero.  The following
statement summarizes Corollary~\ref{cor-smooth-truncated-blowup},
Proposition~\ref{prop-constant-sheaf-blowup}, and
Theorem~\ref{thm-generalized-bc-blowup}.

\begin{mthm}\label{main-derived-blowup}
Let $\pi:\tilde{X}\longrightarrow X$ be the blow-up of a complex
manifold along a closed complex submanifold $\iota:Z\hookrightarrow X$
of codimension $c\geq2$.  No compactness assumption is imposed.
\begin{enumerate}[(i)]
\item\label{main-derived-blowup-decompositions} There are derived decompositions
\[
R\pi_*S_{\tilde{X}}\simeq S_X\oplus
\bigoplus_{r=1}^{c-1}\iota_*S_Z[-2r]
\]
in $D^b(S_X)$ and, for integers $p\leq q$,
\[
R\pi_*\Omega_{\tilde{X}}^{[p,q]}
\simeq\Omega_X^{[p,q]}\oplus
\bigoplus_{r=1}^{c-1}\iota_*\Omega_Z^{[p-r,q-r]}[-2r]
\]
in $D^b(\mathbb C_X)$, together with the antiholomorphic analogue.
\item\label{main-derived-blowup-compatibility} The decompositions can be chosen compatibly with the maps
induced by $f$ and $g$, yielding an isomorphism
\[
R\pi_*\mathcal B_{\tilde{X}}^{p,q}(f,g)
\simeq\mathcal B_X^{p,q}(f,g)\oplus
\bigoplus_{r=1}^{c-1}\iota_*\mathcal B_Z^{p-r,q-r}(f,g)[-2r]
\]
in $D^b(\mathbb Z_X)$ for every $p,q\in\mathbb Z$.
This isomorphism is compatible with the natural projections from the
Bott--Chern complexes to their coefficient complexes.
\end{enumerate}
\end{mthm}

The point of this formulation is the compatibility in
(\ref{main-derived-blowup-compatibility}), rather than the separate existence
of the decompositions in (\ref{main-derived-blowup-decompositions}).
Taking hypercohomology gives the corresponding blow-up formulae in all
degrees.  For compact $X$, the choice $S=\mathbb C$ and
$f=g=\mathrm{id}_{\mathbb C}$ recovers the formula of S.~Yang--X.~Yang
\cite[Theorem 3.7]{SY}; the choice $S=\mathbb Z$ with both maps the
natural inclusion recovers the integral formulae of Y.~Chen--S.~Yang and X.~Wu
cited above.  In contrast to these particular coefficient choices, the
statement allows an arbitrary $S$ and two possibly different maps to
$\mathbb C$.

The derived formulation also clarifies the role of coefficients in the
Dolbeault formula.  Its single-degree holomorphic version, combined
with the derived projection formula, gives the following consequence;
see Remark~\ref{rem-derived-dolbeault-coefficients}.

\begin{cor}\label{cor-main-derived-dolbeault}
In the setting of Theorem~\ref{main-derived-blowup}, let
$M\in D^b(\mathcal O_X)$.  For every integer $p$, there is an
isomorphism in $D^b(\mathcal O_X)$:
\[
\begin{aligned}
R\pi_*\bigl(\Omega_{\tilde{X}}^p
\otimes_{\mathcal O_{\tilde{X}}}L\pi^*M\bigr)
&\simeq (\Omega_X^p\otimes_{\mathcal O_X}M)
\oplus\bigoplus_{r=1}^{c-1}
\iota_*\bigl(\Omega_Z^{p-r}\otimes_{\mathcal O_Z}L\iota^*M\bigr)[-r].
\end{aligned}
\]
\end{cor}

For a holomorphic vector bundle in degree zero, derived and ordinary
pullbacks agree.  Thus this statement recovers, after taking
hypercohomology, the bundle-valued formula of Rao--S.~Yang--X.~Yang in the
compact case and Meng's extension without compactness.  It also gives
a formulation for bounded derived coefficients.  This distinction is
necessary when comparing with coherent sheaves: the classical formula
with ordinary pullbacks does not extend to them in general, as our
counterexample in \S~\ref{counterexample-part} shows.

To describe a further consequence of
Theorem~\ref{main-derived-blowup}.(\ref{main-derived-blowup-compatibility}),
suppose now that $X$ is compact and
write
\[
H^{p,q}_{\mathrm{BC}}(X;f,g)
:=\mathbb H^{p+q}(X,\mathcal B_X^{p,q}(f,g)),\qquad
\mathcal K_X^{p,q}(f,g):=
\ker\!\left(H^{p,q}_{\mathrm{BC}}(X;f,g)
\xrightarrow{\delta_{p,q}^*}H^{p+q}(X;S)\right),
\]
where $\delta_{p,q}$ is the projection to $S_X$.
We introduce the \emph{$\partial\bar\partial$-property relative to
$(f,g)$} by requiring all these comparison maps to be injective.
For $S=\mathbb C$ and $f=g=\mathrm{id}_{\mathbb C}$, this is the
usual $\partial\bar\partial$-lemma.  The compatible derived splittings of
Theorem~\ref{main-derived-blowup}.(\ref{main-derived-blowup-compatibility}) give
the following refinement of the cohomological blow-up formula in
Proposition~\ref{prop-generalized-ddbar-blowup}.

\begin{cor}\label{cor-main-relative-blowup}
In the setting of Theorem~\ref{main-derived-blowup}, assume that $X$
is compact.  Then, for every $p,q\in\mathbb Z$,
\[
\mathcal K_{\tilde{X}}^{p,q}(f,g)
\cong\mathcal K_X^{p,q}(f,g)\oplus
\bigoplus_{r=1}^{c-1}\mathcal K_Z^{p-r,q-r}(f,g).
\]
Consequently, $\tilde{X}$ satisfies the $\partial\bar\partial$-property
relative to $(f,g)$ if and only if both $X$ and $Z$ satisfy it.
\end{cor}

We also characterize this relative property using the Fr\"olicher
spectral sequence and the Hodge filtrations in
Theorem~\ref{thm-relative-ddbar-criterion}.  When $f$ and $g$ are
surjective with different kernels, it is equivalent to degeneration
at $E_1$ by Theorem~\ref{thm-relative-ddbar-criterion}.(\ref{thm-relative-ddbar-distinct-kernels}).
When their kernels agree, Theorem~\ref{thm-relative-ddbar-criterion}.(\ref{thm-relative-ddbar-equal-kernels})
imposes an additional condition on the Hodge filtrations, twisted by
the field automorphism relating $f$ and $g$.  Thus the ordered
coefficient maps distinguish spectral sequence degeneration from the
stronger comparison-map condition.

Three counterexamples clarify the scope of these results.  First, the
Dolbeault blow-up formula with ordinary pullback fails for a torsion-free
coherent sheaf on $\mathbb P^2$.  Second, the projective cone over a
positive-genus curve shows that the de Rham blow-up formula does not
extend unchanged to singular ambient spaces.  Third, Bott--Chern
cohomology does not admit the naive tensor-product K\"unneth decomposition.
Together, these examples highlight the role of the coefficient
hypotheses, the smoothness of the ambient space, and the structure of
the underlying complexes.

The paper is organized as follows.  Section~\ref{preliminaries} reviews
bounded double complexes, derived-category constructions, currents,
Bott--Chern complexes, and the $\partial\bar\partial$-lemma.
In Section~\ref{threefold-part}, we prove Theorem~\ref{main-thm} and
Corollary~\ref{cor-alessandrini-negative}, obtaining counterexamples to
heredity and to invariance under modifications, respectively.
Section~\ref{derived-part} develops the sheaf-level decompositions and
establishes the derived blow-up formulae with coefficients.
Section~\ref{sec-generalized-ddbar} introduces the relative
$\partial\bar\partial$-property, gives its characterization in terms of
the Fr\"olicher spectral sequence and Hodge filtrations, and proves the
blow-up formula for the kernel of the comparison map.
Section~\ref{counterexample-part} examines the limitations of the
classical blow-up formulae for coherent coefficients and singular
ambient spaces, as well as the failure of the naive Bott--Chern
K\"unneth formula. 
Appendix~\ref{appendix-homological-counterexample} identifies an
additional injectivity hypothesis needed in a homological-algebra
argument used in earlier blow-up proofs and verifies it in the
Bott--Chern hypercohomology setting.

Throughout the paper, unless otherwise stated, all complex manifolds
are connected, and the term \emph{submanifold} means a closed embedded
complex submanifold.  The letters $i,j,k,l,m$, as well as
$p,q,r,s,t,c$, usually denote integers, whereas $n$ usually denotes the complex
dimension of the complex manifold under consideration.

\subsection*{Acknowledgments}
The authors would like to express their sincere gratitude to Professors
Song Yang, Xiangdong Yang, Lingxu Meng, D. Angella and J. Stelzig for many stimulating discussions
and valuable insights concerning blow-up formulae.  They are also grateful
to Professors Mingchen Xia and Zhitong Su for their generous assistance and
insightful suggestions in constructing the counterexamples.
Finally, the authors thank Professor N.~Honda for his kind interest in this work.

\section{Preliminaries: double complexes, currents, and Bott--Chern complexes}
\label{preliminaries}

This section collects the algebraic, analytic, and cohomological tools used
in the rest of the paper.  We begin with bounded double complexes and the
functors that extract the relevant cohomology theories, then record the
projection formulae used in the derived constructions.  We next review
currents and Bott--Chern complexes, and conclude with numerical and
deformation-theoretic criteria for the $\partial\bar\partial$-lemma and the
degree-one obstruction needed in the geometric construction.

\subsection{Bounded double complexes}

This subsection fixes the homological-algebra language used throughout the
paper.  We recall bounded double complexes and their Fr\"olicher spectral
sequences, introduce $E_r$-isomorphisms, and describe the totalization and
truncation functors that will later convert double-complex decompositions
into de Rham, Dolbeault, and Bott--Chern statements.

Fix a field $K$.  Recall that a \emph{double complex} $A$ over $K$ consists of
$K$-vector spaces $\{A^{s,t}\}_{s,t\in\mathbb Z}$ together with morphisms
\[
\partial_1:A^{s,t}\longrightarrow A^{s+1,t},
\qquad
\partial_2:A^{s,t}\longrightarrow A^{s,t+1},
\]
such that
\[
\partial_1^2=\partial_2^2=0,
\qquad
\partial_1\partial_2+\partial_2\partial_1=0.
\]
We say that $A$ is \emph{bounded} if $A^{s,t}\neq0$ for only finitely many pairs
$(s,t)$.
We use square brackets for \emph{shifts} of both double complexes and
cochain complexes. For a double complex $A$ and an integer $r$, define $A[r]$ to be a double complex by setting
\[
A[r]^{p,q}:=A^{p+r,q+r},
\]
with the induced differentials.

\begin{Ex}[Zigzags]\label{ex-zigzag}\normalfont
A \emph{zigzag} is a double complex with one-dimensional components at the
vertices of a finite path in the bidegree lattice and zero components
elsewhere. The edges of the path alternate between horizontal and vertical
unit segments, and the corresponding nonzero differentials are
isomorphisms. Along the path, the arrows alternate in direction, so that
at each interior vertex both arrows enter or both arrows leave; all other
differentials are zero. The \emph{length} of a zigzag is the number of its
nonzero components.
\end{Ex}

\begin{Ex}[Squares]\label{ex-square}\normalfont
A \emph{square} is a double complex whose only nonzero components are
one-dimensional spaces in bidegrees $(p,q)$, $(p+1,q)$, $(p,q+1)$, and
$(p+1,q+1)$, with all four edge maps isomorphisms satisfying
$\partial_1\partial_2+\partial_2\partial_1=0$.
For example, choose generators $u$, $v$, $w$, and $z$ in these respective
bidegrees and set
\[
\partial_1u=v,\qquad \partial_2u=w,\qquad
\partial_1w=z,\qquad \partial_2v=-z,
\]
with all other differentials zero.
\end{Ex}

Denote the categories of bounded double complexes and complexes over $K$ by
$\mathrm{DC}_K^b$ and $\mathrm{Ch}(K)$, respectively. 
For $A\in\mathrm{DC}_K^b$, define
\[
H_{\partial_2}^{p,q}(A):=
\frac{\ker(\partial_2:A^{p,q}\longrightarrow A^{p,q+1})}
{\operatorname{im}(\partial_2:A^{p,q-1}\longrightarrow A^{p,q})},
\qquad
H_{\partial_1}^{p,q}(A):=
\frac{\ker(\partial_1:A^{p,q}\longrightarrow A^{p+1,q})}
{\operatorname{im}(\partial_1:A^{p-1,q}\longrightarrow A^{p,q})}.
\]
The \emph{total complex} $A_{\mathrm{tot}}$ is defined by
\[
A_{\mathrm{tot}}^m:=\bigoplus_{p+q=m}A^{p,q},
\qquad d:=\partial_1+\partial_2.
\]
It carries two
\emph{Hodge filtrations}, $F_1^\cdot$ and $F_2^\cdot$:
\[
F_1^k A_{\mathrm{tot}}^m:= \bigoplus_{\substack{p+q=m \\ p \geq k}} A^{p,q}, \qquad
F_2^k A_{\mathrm{tot}}^m:= \bigoplus_{\substack{p+q=m \\ q \geq k}} A^{p,q}.
\]
The two filtrations induce the \emph{Fr\"olicher spectral sequences}:
\begin{align*}
S_1:\quad {}_1E_1^{p,q} &= H_{\partial_2}^{p,q}(A) \Longrightarrow (H^{p+q}(A_{\mathrm{tot}}), F_1), \\
S_2:\quad {}_2E_1^{p,q} &= H_{\partial_1}^{p,q}(A) \Longrightarrow (H^{p+q}(A_{\mathrm{tot}}), F_2).
\end{align*}

For $r\in\mathbb N\cup\{\infty\}$, a morphism
$f:A\longrightarrow B$ of bounded double complexes is called an
\emph{$E_r$-isomorphism} if it induces an isomorphism on the $E_r$-pages of both
Fr\"olicher spectral sequences.

Similarly, one can define the category of bounded double complexes of
sheaves of $K$-vector spaces on a topological space $X$.
A morphism $\varphi:\mathcal A\longrightarrow\mathcal B$ of bounded double
complexes of sheaves of $K$-vector spaces is called
an \emph{$E_r$-isomorphism} if, for every $x\in X$, the induced morphism
$\varphi_x:\mathcal A_x\longrightarrow\mathcal B_x$ is an $E_r$-isomorphism
of double complexes over $K$.

\begin{Pro}[{\!\cite[Proposition~12]{S1}}]\label{prop-double-complex-functor}
Let $f:A\longrightarrow B$ be an $E_r$-isomorphism of bounded double
complexes with $r\geq1$.
Then for any linear functor
$L:\mathrm{DC}_K^b\longrightarrow\mathrm{Vect}_K$
that sends every square and every even-length zigzag of length \(< 2r\) to the zero vector space,
the induced map $L(f):L(A)\longrightarrow L(B)$ is an isomorphism.
\end{Pro}

In particular, if $r=1$, the condition in this proposition is equivalent to
the following: $f$ induces isomorphisms
$H_{\partial_1}^{s,t}(A)\to H_{\partial_1}^{s,t}(B)$ and
$H_{\partial_2}^{s,t}(A)\to H_{\partial_2}^{s,t}(B)$ for all
$s,t\in\mathbb N$, and $L(S)=0$ for every square $S$.

\begin{Ex}\label{ex-total-complex}\normalfont
Define $F: \mathrm{DC}_{K}^{b}\longrightarrow \mathrm{Ch}(K)$ to be a functor by
\[
F(A)^l \;:=\; \bigoplus_{\substack{i+j=l}} A^{i,j},
\]
with differential \(d= \partial_1+\partial_2: F(A)^l \to F(A)^{l+1}\).
The complex $F(A)$ is the \emph{total complex} associated with $A$.
\end{Ex}

\begin{Ex}\label{ex-Lpq-functor}\normalfont
For $p,q\in\mathbb Z$, let $L^{p,q}:\mathrm{DC}_{\mathbb C}^b\longrightarrow\mathrm{Ch}(\mathbb C)$
be the linear functor defined as follows.
For any $A\in \mathrm{DC}_{\mathbb{C}}^{b}$,   $L^{p,q}(A)$ is the following complex:
\[
L^{p,q}(A)^l:=
\begin{cases}
\displaystyle \bigoplus_{\substack{s+t=l \\ s<p,\; t<q}} A^{s,t}, & l \le p+q-2, \\[1.2em]
\displaystyle \bigoplus_{\substack{s+t=l+1 \\ s\ge p,\; t\ge q}} A^{s,t}, & l \ge p+q-1.
\end{cases}
\]
The differentials of $L^{p,q}$ are $d=\partial_1+\partial_2$, except at
degree $p+q-2$, where the differential is $\partial_1\partial_2$.
This functor can be defined in the same way on the category of double
complexes of sheaves on a complex manifold $X$; the resulting functor is
still denoted by $L^{p,q}$.
If $\mathcal E_X$ denotes the double complex whose component in bidegree
$(s,t)$ is the sheaf of germs of $(s,t)$-forms on $X$, then the complex
$L^{p,q}(\mathcal E_X)$ will be denoted below by $\mathcal L_X^{p,q}$.
For every square $S$, the complex $L^{p,q}(S)$ is acyclic; that is, all its
cohomology groups vanish; see \cite[Lemma 2.3]{Se25}.
\end{Ex}
\begin{Ex}\label{ex-Tpq-functor}\normalfont
Define the functor
$T^{p,q}:\mathrm{DC}_{\mathbb C}^b\longrightarrow\mathrm{Ch}(\mathbb C)$ by
\[
T^{p,q}(A)^l \;:=\; \bigoplus_{\substack{s+t=l\\ p\le s\le q}} A^{s,t},
\]
with differential 
 \[
d\alpha=
\begin{cases}
\partial_1 \alpha+\partial_2 \alpha, & p\le s<q,\\[2mm]
\partial_2 \alpha, & s=q,
\end{cases}
\qquad \alpha\in A^{s,t}.
\]
For every square $S$, the complex $T^{p,q}(S)$ is acyclic.
For integers $p,q$, let $\Omega_X^{[p,q]}$ denote the \emph{truncated holomorphic
de Rham complex} whose degree-$k$ term is
\[
(\Omega_X^{[p,q]})^k:=
\begin{cases}
\Omega_X^k, & p\le k\le q,\\
0, & \text{otherwise},
\end{cases}
\]
with differential induced by $\partial$, where $\Omega_X^k=0$ for
$k<0$ or $k>\dim_{\mathbb C}X$.
Note that a truncation with $p>q$ is zero. By the Dolbeault lemma, the natural
inclusion of complexes $\Omega_X^{[p,q]}\longrightarrow T^{p,q}(\mathcal E_X)$ is a
quasi-isomorphism.
\end{Ex}

\subsection{Derived category}
\label{subsec-derived-category-preliminaries}

We review some notions and basic properties of derived categories, following \cite{KS}. Let $f:Y\longrightarrow X$ be a continuous map of locally compact
Hausdorff spaces, and $\mathcal{R}_X$ a sheaf of commutative rings on $X$.
For a sheaf $G$ of $f^{-1}\mathcal{R}_X$-modules, its \emph{direct image} $f_*G$ and
\emph{direct image with proper support} $f_!G$ are defined, for every open
subset $U\subset X$, by
\[
\begin{aligned}
(f_*G)(U)&:=\Gamma(f^{-1}U,G),\\
(f_!G)(U)&:=\{s\in\Gamma(f^{-1}U,G):
f|_{\operatorname{supp}(s)}:\operatorname{supp}(s)\to U
\text{ is proper}\}.
\end{aligned}
\]
These are left exact functors from sheaves of $f^{-1}\mathcal{R}_X$-modules to
sheaves of $\mathcal{R}_X$-modules. If $f$ is proper, then $f_!=f_*$.

Let $D^b(\mathcal{R}_X)$ denote the \emph{bounded derived category} of sheaves of
$\mathcal{R}_X$-modules on $X$, obtained from the homotopy category of bounded complexes
of $\mathcal{R}_X$-modules by
inverting quasi-isomorphisms. Similarly, $D^+(\mathcal{R}_X)$
(resp.\ $D^-(\mathcal{R}_X)$) denotes the \emph{bounded-below derived category}
(resp.\ \emph{bounded-above derived category}). The \emph{right derived functors} of $f_*$ and
$f_!$ are
\[
Rf_*,Rf_!:D^+(f^{-1}\mathcal{R}_X)\longrightarrow D^+(\mathcal{R}_X).
\]
For an injective resolution $G\to I$, they are represented by
\[
Rf_*G=f_*I,
\qquad Rf_!G=f_!I.
\]

The \emph{derived tensor product} is the bifunctor
\[
-\otimes^L_{\mathcal{R}_X}-:
D^-(\mathcal{R}_X)\times D^-(\mathcal{R}_X)
\longrightarrow D^-(\mathcal{R}_X)
\]
obtained by deriving the tensor product of sheaves of
$\mathcal{R}_X$-modules. For $F,G\in D^-(\mathcal{R}_X)$,
choose a quasi-isomorphism $P\to F$ with $P$ a
bounded-above complex of flat $\mathcal{R}_X$-modules. Then
\[
(F\otimes^L_{\mathcal{R}_X}G)^n
:=\bigoplus_{i+j=n}P^i\otimes_{\mathcal{R}_X}G^j,
\]
with differential
$d(u\otimes v)=d_Pu\otimes v+(-1)^i u\otimes d_Gv$ for $u\in P^i$ and $v\in G^j$.

For a commutative ring $R$, its \emph{weak global dimension} is
\[
\mathrm{wgld}(R):=\sup_M\operatorname{fd}_R(M),
\]
where $M$ ranges over all $R$-modules and $\operatorname{fd}_R(M)$ is the
\emph{flat dimension} of $M$, namely the least length of a flat resolution,
or $\infty$ if no finite flat
resolution exists. For a sheaf of rings $\mathcal{R}_X$, set
\[
\mathrm{wgld}(\mathcal{R}_X):=\sup_{x\in X}\mathrm{wgld}((\mathcal{R}_X)_x).
\]

We shall use the following projection formula for derived direct images with
proper support.

\begin{Pro}[{\cite[Proposition 2.6.6]{KS}}]
Suppose that $f:Y\longrightarrow X$ is a continuous map of locally compact
spaces and that $\mathcal{R}_X$ is a sheaf of commutative rings on $X$ with
$\mathrm{wgld}(\mathcal{R}_X)<\infty$.  For
$G\in D^b(f^{-1}\mathcal{R}_X)$ and $F\in D^b(\mathcal{R}_X)$, there is a natural
isomorphism in $D^+(\mathcal{R}_X)$:
\[Rf_!G\otimes^L_{\mathcal{R}_X}F \xrightarrow{\sim} Rf_!(G\otimes^L_{f^{-1}\mathcal{R}_X}f^{-1}F).  \]
\end{Pro}

As a consequence, we obtain the following projection formula for proper
holomorphic maps.
\begin{Cor}\label{cor-derived-projection-formula}
Suppose that $f:Y\longrightarrow X$ is a proper holomorphic map of complex
manifolds.  For $G\in D^b(\mathcal O_Y)$ and $F\in D^b(\mathcal O_X)$,
there is a natural isomorphism in $D^b(\mathcal O_X)$:
\[Rf_*G\otimes^L_{\mathcal{O}_X}F \xrightarrow{\sim} Rf_*(G\otimes^L_{\mathcal{O}_Y}Lf^{*}F),  \]
where
\[
Lf^*F:=\mathcal O_Y\otimes^L_{f^{-1}\mathcal O_X}f^{-1}F
\]
is the \emph{derived pullback} of $F$.
\end{Cor}

Let $g:F\longrightarrow G$ be a morphism of bounded complexes of sheaves
of $\mathcal R_X$-modules.  Its \emph{mapping cone} $\operatorname{Cone}(g)$
is the complex defined by
\[
\operatorname{Cone}(g)^n:=G^n\oplus F^{n+1},
\qquad
d_{\operatorname{Cone}(g)}^n:=
\begin{pmatrix}
d_G^n & g^{n+1}\\
0 & -d_F^{n+1}
\end{pmatrix}.
\]
The natural inclusion and projection give a distinguished triangle
\[
F\xrightarrow{g}G\longrightarrow\operatorname{Cone}(g)
\longrightarrow F[1]
\]
in $D^b(\mathcal R_X)$.

\subsection{Spaces of currents}
We next recall the spaces of currents needed to define pushforward morphisms
for proper holomorphic maps.  We fix the relevant topologies and bigrading
conventions, and define the pushforward of currents.  For detailed
background, we refer the reader to \cite[\S~1]{King} and \cite[\S~I.2]{De}.

Let $X$ be a complex manifold of dimension $n$, and denote by $A^{p,q}(X)$ the
vector space of smooth complex-valued $(p,q)$-forms on $X$.  On a coordinate
neighborhood $\Omega\subset X$, an element $u\in A^{p,q}(X)$ can be written as
\[
u = \sum_{\substack{I,J \subset \{1,\dots,n\} \\ |I| = p, |J| = q}}
u_{IJ} \, dz_I\wedge d\bar{z}_J.
\]
For every compact subset $L\subset\Omega$ and integer $s\in\mathbb N$, define
the \emph{seminorm}
\[
p_{L,\Omega}^s(u) = \sup_{x\in L} \max_{\substack{I,J \\ |\alpha|+ |\beta|\le s}} \left|D^{\alpha\beta}u_{IJ}(x)\right|,
\]
where $\alpha=(\alpha_1,\dots,\alpha_n)$ and
$\beta=(\beta_1,\dots,\beta_n)$ run over $\mathbb N^n$, and
\[
D^{\alpha\beta} = \frac{\partial^{|\alpha|+|\beta|}}{\partial z_1^{\alpha_1}\cdots \partial z_n^{\alpha_n}\,\partial\bar z_1^{\beta_1}\cdots \partial\bar z_n^{\beta_n}}.
\]
Here $|\alpha|=\alpha_1+\cdots+\alpha_n$ and
$|\beta|=\beta_1+\cdots+\beta_n$.

Endow $A^{p,q}(X)$ with the topology induced by the seminorms
$p_{L,\Omega}^s$ as $s$, $L$, and $\Omega$ vary, and denote the resulting
topological vector space by $\mathcal E^{p,q}(X)$.

For a compact subset $K\subset X$, denote by $\mathcal E^{p,q}(K)$ the
subspace of forms supported in $K$.  The \emph{space of compactly supported forms}
is then
$\mathcal E_c^{p,q}(X)=\bigcup_K\mathcal E^{p,q}(K)$.
\begin{De}
The space $\mathcal{DE}^{p,q}(X)$ of \emph{currents of type $(p,q)$} consists of
all linear maps
$T:\mathcal{E}_c^{n-p,n-q}(X)\longrightarrow\mathbb C$ such that, for every
compact set $K$, the restriction of $T$ to
$\mathcal{E}^{n-p,n-q}(K)$ is continuous.
\end{De}

There are two operators $d'$ and $d''$ on $\mathcal{DE}^{p,q}(X)$:
\[
d':\mathcal{DE}^{p,q}(X)\longrightarrow\mathcal{DE}^{p+1,q}(X),
\qquad T\longmapsto d'T,\qquad d'T(u)=T(\partial u),
\]
and
\[
d'':\mathcal{DE}^{p,q}(X)\longrightarrow\mathcal{DE}^{p,q+1}(X),
\qquad T\longmapsto d''T,\qquad d''T(u)=T(\bar\partial u).
\]

If $f:Y\longrightarrow X$ is a proper holomorphic map of complex manifolds,
set $\delta_f=\dim_{\mathbb C}Y-\dim_{\mathbb C}X$.  The \emph{pushforward of
currents} is the map
\[
f_\#: \mathcal{DE}^{p+\delta_f,q+\delta_f}(Y)
\longrightarrow\mathcal{DE}^{p,q}(X),
\qquad (f_\#T)(u)=T(f^*u),
\]
for $u\in\mathcal E_c^{n-p,n-q}(X)$, where $n=\dim_{\mathbb C}X$.

\subsection{Bott--Chern complexes}

This subsection introduces the sheaf complexes that realize Bott--Chern
cohomology in the derived category.  We recall the definition from
\cite[\S~3]{CSY} and \cite[Definition~1]{Wu} for $R=\mathbb Z$ and
\cite[Definition~2.2]{SY} and \cite[\S~2.2]{YY20} for $R=\mathbb C$,
and relate these complexes to Deligne complexes and truncated holomorphic
de Rham complexes.
\begin{De}\label{def-bc-coefficient-ring}
Fix $p,q\in\mathbb N$ and a subring $R\subset\mathbb C$.
The \emph{$R$-Bott--Chern complex} $\mathcal B_X^{p,q}(R)$ of type $(p,q)$ on a
complex manifold $X$ is defined by
\[
\mathcal B_X^{p,q}(R):=
\operatorname{Cone}\!\left(R_X\xrightarrow{(+,-)}
\Omega_X^{[0,p-1]}\oplus\overline{\Omega_X^{[0,q-1]}}\right)[-1],
\]
where $R_X$ is the constant sheaf with stalk $R$, placed in degree $0$.
The chain map sends a local section $r$ to $(r,-r)$ in degree $0$,
with the components corresponding to zero truncations omitted.
\end{De}

The complex is explicitly
{\small
\[
0 \longrightarrow R_X
\stackrel{(+,-)}{\longrightarrow} \mathcal O_X\oplus\overline{\mathcal O_X}
\stackrel{(\partial,\bar\partial)}{\longrightarrow}
\Omega_X^1\oplus\overline{\Omega_X^1}
\longrightarrow\cdots
\longrightarrow\Omega_X^{p-1}\oplus\overline{\Omega_X^{p-1}}
\stackrel{(0,\bar\partial)}{\longrightarrow}\overline{\Omega_X^p}
\stackrel{\bar\partial}{\longrightarrow}\cdots
\longrightarrow\overline{\Omega_X^{q-1}}\longrightarrow 0.
\]
}

If $q=0$, $\mathcal B_X^{p,q}(R)$ reduces to the \emph{Deligne complex}
$R_D(p)$, given by
\[
R_D(p):=0\longrightarrow R_X\hookrightarrow\mathcal O_X
\xrightarrow{\partial}\Omega_X^1\xrightarrow{\partial}\cdots
\xrightarrow{\partial}\Omega_X^{p-1}\longrightarrow0,
\]
with $R_X$ in degree $0$ under our coefficient convention;
cf.\ \cite[\S~3]{Wu}.

The following definition generalizes Definition~\ref{def-bc-coefficient-ring}
to an arbitrary coefficient ring.
\begin{De}\label{def-generalized-bc-complex}
Let $S$ be a commutative ring (not necessarily a subring of
$\mathbb C$) with unit, and $f,g:S\longrightarrow\mathbb C$
morphisms of rings. For integers $p,q$, the \emph{generalized Bott--Chern
complex} associated with the ordered pair $(f,g)$ is
\[
\mathcal B_X^{p,q}(f,g):=
\operatorname{Cone}\!\left(S_X\xrightarrow{\Delta_{f,g}}
\Omega_X^{[0,p-1]}\oplus\overline{\Omega_X^{[0,q-1]}}\right)[-1],
\]
where $\Delta_{f,g}$ denotes the chain map induced in degree $0$ by
\[
S_X\xrightarrow{(f,g)}\mathbb C_X\oplus\mathbb C_X
\hookrightarrow\mathcal O_X\oplus\overline{\mathcal O_X},
\]
with the components corresponding to zero truncations omitted.
The chain map is zero in all other degrees.
\end{De}

Set $\mathcal L_X^{p,q}:=L^{p,q}(\mathcal E_X)$ as in  Example~\ref{ex-Lpq-functor},
where $\mathcal E_X$ is the double complex of sheaves of smooth forms on $X$.
The following proposition shows that
$\mathcal L_X^{p,q}[-1]$ is a resolution of
$\mathcal{B}_X^{p,q}(\mathbb C)$.

\begin{Pro}[{\cite[Propositions~4.2--4.4]{Sch}}]
The natural morphism
$\mathcal{B}_X^{p,q}(\mathbb C)\longrightarrow\mathcal L_X^{p,q}[-1]$ is a
quasi-isomorphism.
\end{Pro}

\subsection{\texorpdfstring{$\partial\bar\partial$}{partial-bar-partial}-lemma and its properties}
We now review the geometric meaning of the
$\partial\bar\partial$-lemma.  Besides recalling its formulation in the
double complex of smooth forms, we record the Bott--Chern--Aeppli numerical
criterion, deformation openness, and the consequence for manifolds in
Fujiki's class $\mc C$.

Let $X$ be a compact complex manifold and let $A^{p,q}(X)$ denote the space of
smooth complex-valued forms of type $(p,q)$.  The complexified de Rham
differential decomposes as $d=\partial+\bar\partial$.

\begin{defn}[$\partial\bar\partial$-lemma]\label{def-ddbar}
The manifold $X$ is said to satisfy the
\emph{$\partial\bar\partial$-lemma} if, for
every form $\alpha\in A^{p,q}(X)$ satisfying $d\alpha=0$, the following
conditions are equivalent:
\begin{equation*}
 \alpha\in\im d,\qquad
 \alpha\in\im\partial,\qquad
 \alpha\in\im\bar\partial,\qquad
 \alpha\in\im(\partial\bar\partial).
\end{equation*}
Equivalently,
\begin{equation*}
 \ker\partial\cap\ker\bar\partial\cap\im d
 =\im(\partial\bar\partial)
\end{equation*}
on forms of every pure type.
\end{defn}

The \emph{Bott--Chern cohomology groups} and
\emph{Aeppli cohomology groups} of $X$ are, respectively,
\begin{align*}
 H_{\mathrm{BC}}^{p,q}(X)
 &:=\frac{\ker\partial\cap\ker\bar\partial\cap A^{p,q}(X)}
 {\im(\partial\bar\partial)\cap A^{p,q}(X)},
\\
 H_{\mathrm A}^{p,q}(X)
 &:=\frac{\ker(\partial\bar\partial)\cap A^{p,q}(X)}
 {(\im\partial+\im\bar\partial)\cap A^{p,q}(X)}.
\end{align*}
The equivalence between Definition~\ref{def-ddbar} and the principle of two
types is proved in \cite[Proposition 5.17 and (5.21)]{DGMS75}.

For $k\in\mb Z$, put
\begin{equation*}
 \Delta^k(X):=\sum_{p+q=k}
 \left(h_{\mathrm{BC}}^{p,q}(X)+h_{\mathrm A}^{p,q}(X)\right)-2b_k(X).
\end{equation*}
Angella--Tomassini proved
\begin{equation}\label{AT-criterion}
 \Delta^k(X)\geq0,\qquad
 X\text{ is }\partial\bar\partial
 \Longleftrightarrow
 \Delta^k(X)=0\text{ for every }k;
\end{equation}
see \cite[Theorems A and B]{AT13}.

\begin{prop}[Openness under deformation]\label{prop-openness}
Let $\pi:\mc X\to B$ be a holomorphic family of compact complex manifolds.
If a fiber $X_{t_0}$ satisfies the $\partial\bar\partial$-lemma, then all
fibers $X_t$ with $t$ sufficiently close to $t_0$ satisfy the
$\partial\bar\partial$-lemma.
\end{prop}

\begin{proof}
The Bott--Chern and Aeppli groups can be realized as kernels of fourth-order
elliptic operators.  Their dimensions are therefore upper semicontinuous in a
smooth family, whereas the Betti numbers are locally constant.  If
$\Delta^k(X_{t_0})=0$, then after shrinking $B$ about $t_0$ one has
\[
 0\leq \Delta^k(X_t)\leq \Delta^k(X_{t_0})=0.
\]
Equation \eqref{AT-criterion} gives the assertion.  This is the numerical
semicontinuity proof of \cite[Corollary 2.7]{AT13}.
\end{proof}

\begin{rem}
The openness of the $\partial\bar\partial$-lemma under small deformations
was already established by J.~Bingener \cite[Theorem~6.5]{Bin83},
using the double-complex result of \cite[Lemma~2.17]{Bin83}.
More precisely, for a proper smooth family over a real analytic base,
the locus of pseudo-K\"ahler fibers is the complement of a real analytic
subset of the base.  Here pseudo-K\"ahler is understood in the
cohomological sense of \cite[$\S$~6.4]{Bin83}, which is equivalent
to the $\partial\bar\partial$-lemma; see also
\cite[Proposition~2.9]{Bin83}.  Further proofs were given by C.~Voisin
\cite[Propositions~9.20 and~9.21]{Vi02} and C.-C.~Wu
\cite[Theorem~5.12]{Wu06}.  Voisin's argument uses the characterization of
the $\partial\bar\partial$-lemma by the degeneration of the Fr\"olicher
spectral sequence at $E_1$ together with the fact that the Hodge filtration
on each de Rham cohomology group defines a Hodge structure.  This
characterization is due to Deligne--Griffiths--Morgan--Sullivan; see
\cite[Proposition 5.17 and (5.21)]{DGMS75}.  The proof above instead uses the
Bott--Chern--Aeppli numerical criterion and upper semicontinuity.
\end{rem}

\begin{prop}\label{prop-moishezon-ddbar}
Every compact complex manifold in Fujiki's class $\mc C$, and in particular
every compact Moishezon manifold, satisfies the
$\partial\bar\partial$-lemma.
\end{prop}

\begin{proof}
By definition, a compact complex manifold belongs to Fujiki's class $\mc C$ if it is
bimeromorphic to a compact K\"ahler manifold.  After resolving the graph of a
bimeromorphic map, it is dominated by a compact K\"ahler manifold through a
proper modification.  The $\partial\bar\partial$-lemma descends under such a
map.  This descent result was first proved by Parshin
\cite[\S\S~4 and 5]{Pr66}; see also \cite[Theorem 5.22]{DGMS75}.  A
Moishezon manifold belongs to class $\mc C$; the corresponding conclusion is
also stated in \cite[Corollary 5.23]{DGMS75}.
\end{proof}

\subsection{An obstruction in degree one}

This subsection isolates the elementary obstruction used to detect the
non-$\partial\bar\partial$ surface in the main construction.  We show that
the $\partial\bar\partial$-lemma forces the first Betti number to be even,
recall that Hopf surfaces have first Betti number one, and verify that this
number is invariant under bimeromorphisms of smooth compact surfaces.

\begin{lemma}\label{odd-b1}
If a compact complex manifold $X$ satisfies the
$\partial\bar\partial$-lemma, then $b_1(X)$ is even.
\end{lemma}

\begin{proof}
The $\partial\bar\partial$-lemma gives the degree-one Hodge decomposition
\begin{equation*}
 H^1_{dR}(X,\mb C)
 =H^{1,0}_{\bar\partial}(X)\oplus H^{0,1}_{\bar\partial}(X).
\end{equation*}
Complex conjugation interchanges the two summands.  Consequently
\[
 b_1(X)=h^{1,0}_{\bar\partial}(X)+h^{0,1}_{\bar\partial}(X)
       =2h^{1,0}_{\bar\partial}(X).
\]
\end{proof}

\begin{defn}[Hopf surface]\label{def-Hopf}
A \emph{primary Hopf surface} is a compact complex surface whose universal
cover is $\mb C^2\setminus\{0\}$ and whose fundamental group is infinite
cyclic, generated by a holomorphic contraction.  A \emph{Hopf surface} is a
finite unramified quotient of a primary Hopf surface.
\end{defn}

The classification of Hopf surfaces gives
\begin{equation}\label{Hopf-b1}
 b_1(H)=1.
\end{equation}
See \cite[Chapter V, \S~18]{BHPV04} and Kodaira's original
description \cite[pp.~240--243]{Kd66}.

\begin{lemma}\label{surface-bimeromorphic-b1}
Let $S_1$ and $S_2$ be bimeromorphic smooth compact complex surfaces.  Then
$b_1(S_1)=b_1(S_2)$.
\end{lemma}

\begin{proof}
A bimeromorphic map between smooth compact complex surfaces admits a common
resolution which factors into blow-ups at points.  The de Rham cohomology
blow-up formula \cite[Theorem 7.31]{Vi02} shows that blowing up a point on a
surface does not change the first cohomology group.  Applying this to both
morphisms from a common resolution proves the assertion; compare
\cite[Chapter III, \S\S~1--4]{BHPV04}.
\end{proof}

\section{A threefold counterexample to heredity of the
\texorpdfstring{$\partial\bar\partial$}{ddbar}-lemma}
\label{threefold-part}

This section constructs the geometric counterexample underlying the main
results.  We first extract from Honda's twistor spaces a
$\partial\bar\partial$-threefold containing a singular divisor
bimeromorphic to a Hopf surface, and then resolve the pair by blow-ups whose
centers preserve the $\partial\bar\partial$-lemma.  The resulting smooth
divisor disproves heredity in dimension three and, after taking products and
blowing up, yields counterexamples to Alessandrini's modification question
in every dimension at least four.

\subsection{Honda's twistor spaces and the Hopf-birational divisor}
\label{Honda-section}

We summarize the part of Honda's construction that supplies the geometric
input for our counterexample.  Starting from a rational elliptic surface, we
pass to a twistor threefold containing a distinguished non-normal divisor
bimeromorphic to a Hopf surface, and then use deformation theory to choose
the ambient threefold inside the $\partial\bar\partial$ locus.

\subsubsection{The surface in the half-anticanonical system}

Fix an integer $n>4$.  Honda begins with a rational elliptic surface $S_0$
obtained as the minimal resolution of a quotient of $C_{\mathrm{ell}}\times\mathbb P^1$, where
$C_{\mathrm{ell}}$ is an elliptic curve and the quotient involution has eight isolated fixed
points.  The induced elliptic fibration
\[
 f_0:S_0\longrightarrow\mathbb P^1
\]
has two singular fibers of type $I_0^*$,
\begin{align}
 f_0^{-1}(0)&=2C_0+C_1+C_2+C_3+C_4,\label{I0-first}\\
 f_0^{-1}(\infty)&=2\bar C_0+\bar C_1+\bar C_2+\bar C_3+\bar C_4.\notag 
\end{align}
Here the bar is induced by a fixed-point-free anti-holomorphic involution.
Moreover, $K_{S_0}^2=0$ and the anticanonical pencil $|K_{S_0}^{-1}|$
induces $f_0$; see \cite[Proposition 2.1]{Hn15}.

Choose distinct points $p_5,\ldots,p_n$ on $C_0$ away from the four nodes of
the fiber \eqref{I0-first}, together with their conjugates
$\bar p_5,\ldots,\bar p_n$ on $\bar C_0$.  Let
\begin{equation*}
 \pi:S\longrightarrow S_0
\end{equation*}
be the blow-up along the center
$Z=\{p_5,\ldots,p_n,\bar p_5,\ldots,\bar p_n\}$.  Write
$E=\pi^{-1}(Z)$ for the exceptional divisor, with components
$C_5,\ldots,C_n,\bar C_5,\ldots,\bar C_n$.  The strict transforms of $C_0$
and $\bar C_0$ satisfy
\begin{equation*}
 C_0^2=\bar C_0^2=2-n,\qquad K_S^2=8-2n.
\end{equation*}
Writing $f=f_0\circ\pi$, the canonical bundle formula becomes
\begin{equation*}
 K_S^{-1}\simeq f^*\mc O_{\mathbb P^1}(1)
 -\sum_{i=5}^{n}(C_i+\bar C_i).
\end{equation*}
Honda deduces
\begin{equation*}
 h^0(S,mK_S^{-1})=
 \begin{cases}
 0,&m\text{ odd},\\
 1,&m\text{ even},
 \end{cases}
\end{equation*}
and identifies the unique member of $|2K_S^{-1}|$; see
\cite[Proposition 2.2 and equations (2.5)--(2.10)]{Hn15}.
\subsubsection{The twistor threefold and the distinguished divisor}

Let $M=n\mathbb P^2$ carry one of the self-dual metrics used by Honda.  Its
\emph{twistor space} is, as a differentiable manifold, the sphere bundle
\begin{equation*}
 \pi:Z=S(\Lambda^-M)\longrightarrow M.
\end{equation*}
The fiber is $S^2\simeq\mathbb P^1$, and the
Atiyah--Hitchin--Singer almost-complex
structure is integrable since the metric is self-dual
\cite[pp.~425--431]{AHS78}.  Thus $Z$ is a smooth compact complex threefold.
The singular space appearing below is a divisor in $Z$, not the twistor space
itself.

The anticanonical bundle of $Z$ has a natural square root
\begin{equation*}
 F:=K_Z^{-1/2},
\end{equation*}
called the \emph{fundamental line bundle}.  Honda constructs $Z$ so that the
surface $S$ above is the unique member of $|F|$.  Adjunction gives
\begin{equation*}
 K_S=(K_Z+S)|_S=(-2F+F)|_S,\qquad F|_S\simeq K_S^{-1}.
\end{equation*}

For $5\leq i\leq n$, the conjugate curves $C_i$ and $\bar C_i$ project to the
same two-sphere in $M$.  If $\xi_i\in H^2(M,\mb Z)$ is its Poincar\'e dual,
put
\begin{equation*}
 \alpha_i:=\pi^*\xi_i\in H^2(Z,\mb Z),\qquad
 A:=\sum_{i=5}^{n}\alpha_i.
\end{equation*}
The vanishing $H^1(Z,\mc O_Z)=H^2(Z,\mc O_Z)=0$ identifies
$\Pic(Z)$ with $H^2(Z,\mb Z)$, and Honda obtains
\begin{equation*}
 \mc O_Z(\alpha_i)|_S\simeq\mc O_S(C_i-\bar C_i).
\end{equation*}
His cohomology calculation gives
\begin{equation*}
 H^0\bigl(Z,F\otimes\mc O_Z(-A)\bigr)\simeq\mb C;
\end{equation*}
see \cite[Proposition 3.6 and equations (3.18)--(3.30)]{Hn15}.  Let
\begin{equation*}
 D_0\in|F-A|
\end{equation*}
be the unique effective divisor.  Its restriction to $S$ is
\begin{equation}\label{D0-restriction}
 D_0|_S=2C_0+C_1+C_2+C_3+C_4.
\end{equation}

\begin{defn}[Ordinary double curve]\label{ordinary-double}
Let $S$ be a reduced complex surface.  A smooth curve $C\subset S$ is called
an \emph{ordinary double curve} if, at a general point of $C$, the germ of
$S$ is analytically isomorphic to
\[
 \{uv=0\}\subset\mb C^3_{u,v,w},
 \qquad C=\{u=v=0\}.
\]
The two local branches may be exchanged by monodromy; therefore a surface
with an ordinary double curve need not be globally reducible.
\end{defn}

The following is the essential surface-theoretic input.

\begin{thm}[Honda]\label{Honda-input}
For every $n\geq5$, the threefold $Z$ above contains a unique divisor
$D_0\in|F-\alpha_5-\cdots-\alpha_n|$ with the following properties:
\begin{enumerate}[$(i)$]
\item $D_0$ is reduced, irreducible and non-normal;
\item $C_0$ is an ordinary double curve of $D_0$;
\item $D_0$ is bimeromorphic to a Hopf surface.
\end{enumerate}
\end{thm}

\begin{proof}
The irreducibility and the restriction formula \eqref{D0-restriction} are
\cite[Proposition 3.7.(ii)]{Hn15}.  Honda takes an equivariant
desingularization of $D_0$, passes to its minimal model, and uses the
classification of compact complex surfaces with a nontrivial $\mb C^*$-action.
The possible minimal models are a Hopf surface and a parabolic Inoue surface.
The configuration of the curve $D_0\cap\overline{D_0}$ excludes the latter, yielding
the Hopf case.  The ordinary double curve and the Hopf-birational conclusion
are \cite[Theorem 3.11.(ii)]{Hn15}.
\end{proof}
\subsubsection{Deformation from a Joyce twistor space}

It remains important that Honda's threefold may be chosen inside an
arbitrarily small deformation of a $\partial\bar\partial$-manifold.  Let
$Z_J$ be the twistor space of a self-dual metric constructed by Joyce on $n\mathbb P^2$
\cite[Theorem~3.3.1]{Jy95}.  Fujiki's
description of torus actions supplies a torus-invariant real member
$S_J\in|F_{Z_J}|$; see \cite{Fj00}.  Honda's projective models show that
$Z_J$ is Moishezon \cite[Lemma 2.7 and Theorem 2.8]{Hn08}.
Proposition~\ref{prop-moishezon-ddbar} therefore gives
\begin{equation}\label{Joyce-ddbar}
 Z_J\text{ satisfies the }\partial\bar\partial\text{-lemma}.
\end{equation}

Honda deforms the blow-up configuration defining $S_J$ and uses the
co-stability of the pair to lift this deformation to the ambient twistor
space.  The relevant obstruction group vanishes:
\begin{equation*}
 H^2\bigl(Z_J,\Theta_{Z_J}(-S_J)\bigr)=0.
\end{equation*}
Horikawa's co-stability theorem then gives a deformation of pairs
\begin{equation*}
 (\mc Z,\mc S)\longrightarrow(B,0),
 \qquad (Z_0,S_0)=(Z_J,S_J),
\end{equation*}
whose sufficiently general nearby fibers are the pairs used in
Theorem~\ref{Honda-input}; see \cite[\S~5.1 and Remark 5.1]{Hn15} and
\cite[Theorem 8.3]{Hr76}.
The deformation parameter can be taken arbitrarily close to $0$, since the
points and infinitely-near directions in Honda's construction are varied by
a local parameter.

Combining \eqref{Joyce-ddbar} with Proposition~\ref{prop-openness}, choose a
nonzero sufficiently small $t\in B$ and write
\begin{equation*}
 Z:=Z_t.
\end{equation*}
Let $D_0\subset Z$ be the distinguished divisor in
Theorem~\ref{Honda-input}.
We have obtained the following conclusion.

\begin{prop}\label{Honda-ddbar-pair}
For every $n\geq5$ there exist a smooth compact complex
$\partial\bar\partial$-threefold $Z$ and an irreducible reduced divisor
$D_0\subset Z$ such that $D_0$ is non-normal and bimeromorphic to a Hopf surface.
\end{prop}

\begin{rem}\label{real-structure-remark}
A \emph{real structure} on a complex manifold $(M,J)$ is an anti-holomorphic
involution $\rho$, equivalently
\[
 \rho^2=\id,\qquad d\rho\circ J=-J\circ d\rho.
\]
For a twistor space it is the antipodal map on every twistor line.  Honda's
real structure is essential for pairing the blow-up data and producing the
conjugate divisor $\overline{D_0}$, but it is additional structure: after the pair
$(Z,D_0)$ has been obtained, the proof of Theorem~\ref{main-thm} is entirely in
the complex analytic category.
\end{rem}
\subsection{Embedded resolution in a complex threefold}
\label{resolution-section}

The distinguished divisor obtained above is singular, so it must be resolved
without destroying the $\partial\bar\partial$ property of the ambient
threefold.  This subsection recalls embedded resolution and shows that, in
dimension three, its nontrivial centers are points or smooth compact curves;
the blow-up criterion can therefore be applied at every stage.

\begin{defn}[Modification and bimeromorphic map]\label{modification-definition}
A \emph{proper modification} $\mu:\tilde{X}\to X$ of complex manifolds is
a proper surjective holomorphic map for which there exists a nowhere dense
analytic subset $Y\subset X$ such that
\[
 \mu:\tilde{X}\setminus\mu^{-1}(Y)\longrightarrow X\setminus Y
\]
is biholomorphic.  Two irreducible compact complex spaces are
\emph{bimeromorphic} if they contain dense open subsets which are
biholomorphic.
\end{defn}

\begin{defn}[Blow-up and strict transform]\label{blowup-definition}
Let $Z$ be a smooth closed complex submanifold of a complex manifold $X$.
The \emph{blow-up}
\[
 \pi:\tilde{X}:=\Bl_Z X\longrightarrow X
\]
is the proper modification whose \emph{exceptional divisor} is
$E:=\mathbb P(N_{Z/X})$.  If $Y'\subset X$ is a reduced analytic subspace not contained
in $Z$, its \emph{strict transform} is
\begin{equation*}
 \Str_{\Bl_Z X}(Y'):=\overline{\pi^{-1}(Y'\setminus Z)}.
\end{equation*}
\end{defn}

\begin{defn}[Embedded resolution]\label{embedded-resolution-definition}
Let $Y$ be a reduced analytic subspace of a complex manifold $X$.  An
\emph{embedded resolution} of the pair $(X,Y)$ is a finite composition of
blow-ups with smooth centers
\begin{equation}\label{embedded-resolution-sequence}
 X_r\xrightarrow{\pi_r}X_{r-1}\longrightarrow\cdots
 \longrightarrow X_1\xrightarrow{\pi_1}X_0=X
\end{equation}
such that the final strict transform $Y_r$ is smooth and $Y_r$ together with
the reduced exceptional divisor has simple normal crossings.  Here a reduced
divisor has \emph{simple normal crossings} if it is locally given by
$z_1\cdots z_r=0$ in holomorphic coordinates.
\end{defn}

Embedded resolution is stronger than an abstract desingularization
$\widetilde Y\to Y$: it resolves the subspace while retaining an embedding of
the resolved strict transform in a controlled modification of the ambient
manifold.  We use the analytic form of canonical desingularization in
\cite[Theorem 1.6 and Remarks 1.7]{BM97}.

\begin{lemma}\label{resolution-centers}
Let $D_0$ be a reduced irreducible divisor in a smooth compact complex
threefold $Z$.  There is an embedded resolution
\begin{equation*}
 \mu:(\widehat Z,D)\longrightarrow(Z,D_0)
\end{equation*}
such that
\begin{enumerate}[(i)]
\item every nontrivial center is a point or a smooth compact complex curve;
\item $D$ is a smooth compact connected divisor in $\widehat Z$;
\item $\mu|_D:D\to D_0$ is bimeromorphic.
\end{enumerate}
\end{lemma}

\begin{proof}
Apply embedded desingularization to $D_0\subset Z$.  Since $Z$ is compact, all
closed centers and all intermediate manifolds are compact.  A blow-up along a
smooth Cartier divisor is an isomorphism, as its ideal sheaf is
invertible.  Such steps may be omitted.  Every remaining center has
codimension at least two in a threefold and is therefore a point or a smooth
curve; disconnected centers may be treated component by component.

The resolution is an isomorphism over a dense open subset of $(D_0)_{\Reg}$.
Hence the final strict transform is a divisor and maps bimeromorphically to
$D_0$.  Since $D_0$ is irreducible, the inverse image of this dense open subset is
irreducible; its closure $D$ is therefore irreducible and, being smooth,
connected.  Properness gives compactness.
\end{proof}

The following blow-up theorem is the cohomological reason that the ambient
threefold retains the $\partial\bar\partial$ property.

\begin{thm}[Blow-up criterion]\label{blowup-criterion}
Let $\pi:\tilde{X}=\Bl_Z X\longrightarrow X$ be the blow-up of a compact complex manifold $X$
along a smooth compact complex submanifold $Z$ of codimension at least two.
Then
\begin{equation*}
 \tilde{X}\text{ is }\partial\bar\partial
 \quad\Longleftrightarrow\quad
 X\text{ and }Z\text{ are }\partial\bar\partial.
\end{equation*}
\end{thm}

The equivalence is stated in \cite[Proposition 3.3]{Mn21}.  As noted in
\cite[Remark 3.4]{Mn21}, Rao--S.~Yang--X.~Yang
\cite[Theorem 1.6]{RYY19}\cite[Theorem 1.2]{RYY20} first understood this
proposition from the viewpoint of the Deligne--Griffiths--Morgan--Sullivan
criterion for the $\partial\bar\partial$-lemma, and S.~Yang--X.~Yang
\cite[Theorem 1.3]{YY20} studied it from the viewpoint of the
Angella--Tomassini characterization in the case of threefolds.  Shortly
thereafter, Angella--Suwa--Tardini--Tomassini
\cite[Theorem 13]{ASTT20} also considered it using \v{C}ech--Dolbeault
cohomology under some additional assumptions.  Eventually, Stelzig
obtained a blow-up formula for Bott--Chern cohomology and stated this result
explicitly in \cite[Corollary 1.40]{Se18}; see also \cite[\S~5]{Se21}.
Subsequently, Rao--Zou \cite[Corollary 5.4]{RZ24} gave another proof
using Varouchas cohomology.

\begin{cor}\label{resolution-preserves-ddbar}
In Lemma~\ref{resolution-centers}, if $Z$ satisfies the
$\partial\bar\partial$-lemma, then so does $\widehat Z$.
\end{cor}

\begin{proof}
A point is K\"ahler.  Every compact complex curve is a compact Riemann surface
and is also K\"ahler.  Hence all centers in Lemma~\ref{resolution-centers}
satisfy the $\partial\bar\partial$-lemma.  Starting from $Z$ and applying
Theorem~\ref{blowup-criterion} inductively to
\eqref{embedded-resolution-sequence}
shows that every $X_j$, and in particular $X_r=\widehat Z$, satisfies the
$\partial\bar\partial$-lemma.
\end{proof}

\begin{rem}\label{bimeromorphic-alternative}
The $\partial\bar\partial$-lemma was proved to be a bimeromorphic invariant of
smooth compact complex threefolds by 
\cite[Theorem 1.3]{YY20}, \cite[Corollary 1.40]{Se18} and \cite[Corollary 5.4]{RZ24}.  Since
$\widehat Z$ and $Z$ are
bimeromorphic, these results give
Corollary~\ref{resolution-preserves-ddbar} without examining the centers.  The
step-by-step argument above is retained to clarify where the
dimension-three hypothesis enters.
\end{rem}
\subsection{Construction and proof of the main theorem}
\label{proof-section}

We now combine Honda's deformation, embedded resolution, and the blow-up
criterion to prove the main theorem.  After establishing the properties of
the ambient threefold and its smooth divisor separately, we derive
higher-dimensional counterexamples to Alessandrini's modification question
and discuss several geometric consequences of the construction.

Fix $n\geq5$ and take the pair
$(Z,D_0)$ of Proposition~\ref{Honda-ddbar-pair}.  Thus $Z$ is a smooth compact
complex $\partial\bar\partial$-threefold and $D_0\subset Z$ is an irreducible
non-normal divisor bimeromorphic to a Hopf surface.  Apply
Lemma~\ref{resolution-centers} and put
\begin{equation}\label{final-pair}
 \mu:\widehat Z\longrightarrow Z,
 \qquad D:=\Str_{\widehat Z}(D_0).
\end{equation}

\begin{prop}\label{ambient-properties}
The manifold $\widehat Z$ is a smooth compact complex threefold satisfying the
$\partial\bar\partial$-lemma, and $D\subset\widehat Z$ is a smooth compact
connected embedded complex surface.
\end{prop}

\begin{proof}
Smoothness and compactness of $\widehat Z$ follow from the construction as a
finite sequence of blow-ups with smooth compact centers.  Its complex
dimension remains three.  Lemma~\ref{resolution-centers} says that $D$ is a
smooth compact connected divisor; hence it is a closed embedded complex
submanifold of dimension two.  Finally,
Corollary~\ref{resolution-preserves-ddbar} gives the
$\partial\bar\partial$ property of
$\widehat Z$.
\end{proof}

\begin{prop}\label{surface-properties}
The surface $D$ in \eqref{final-pair} has $b_1(D)=1$ and does not satisfy the
$\partial\bar\partial$-lemma.
\end{prop}

\begin{proof}
The restriction $D\to D_0$ of $\mu$ is bimeromorphic by
Lemma~\ref{resolution-centers}.  Theorem~\ref{Honda-input} gives a
bimeromorphic map
from $D_0$ to a Hopf surface $H$.  Thus $D$ and $H$ are bimeromorphic smooth
compact complex surfaces.  Lemma~\ref{surface-bimeromorphic-b1} and
\eqref{Hopf-b1} imply
\begin{equation}\label{b1-final-surface}
 b_1(D)=b_1(H)=1.
\end{equation}
If $D$ satisfied the $\partial\bar\partial$-lemma, Lemma~\ref{odd-b1} would
force $b_1(D)$ to be even, contradicting \eqref{b1-final-surface}.
\end{proof}

\begin{proof}[Proof of Theorem~\ref{main-thm}]
Take $(\widehat Z,D)$ as in \eqref{final-pair}.
Proposition~\ref{ambient-properties} verifies compactness, smoothness, the dimension
count, the $\partial\bar\partial$ property of the ambient threefold, and the
embeddedness of the surface.  Proposition~\ref{surface-properties} proves
that the surface is not a $\partial\bar\partial$-manifold.  This proves all
the assertions.
\end{proof}

\begin{rem}\label{rem-heredity-sharpness}
The ambient dimension in Theorem~\ref{main-thm} is optimal.  Indeed, every
proper complex submanifold of a complex manifold of dimension at most two is
a discrete set or a complex curve, and every compact complex curve is
K\"ahler.  Thus the heredity problem has an affirmative answer in dimensions
at most two and first fails in dimension three.
\end{rem}

\begin{proof}[Proof of Corollary~\ref{cor-alessandrini-negative}]
Fix $n\geq4$ and let $(\widehat Z,D)$ be the pair given by
Theorem~\ref{main-thm}.  Choose a linear inclusion
$\mathbb P^{n-4}\subset\mathbb P^{n-3}$ and set
\[
X^{n}:=\widehat Z\times\mathbb P^{n-3},
\qquad
Z^{n-2}:=D\times\mathbb P^{n-4}\subset X^{n}.
\]
The $\partial\bar\partial$-lemma is stable under products.  Since
$\widehat Z$ satisfies the lemma and $\mathbb P^{n-3}$ is K\"ahler, $X^{n}$
is a $\partial\bar\partial$-manifold.  The K\"unneth formula and
Proposition~\ref{surface-properties} give
\[
b_1(Z^{n-2})=b_1(D)+b_1(\mathbb P^{n-4})=1.
\]
Lemma~\ref{odd-b1} therefore shows that $Z^{n-2}$ does not satisfy the
$\partial\bar\partial$-lemma.  Moreover, $Z^{n-2}$ is a smooth compact complex
submanifold of codimension two in $X^{n}$.

Let
\[
\pi:\tilde{X}^{n}:=\Bl_{Z^{n-2}}X^{n}\longrightarrow X^{n}
\]
be the blow-up along $Z^{n-2}$.  If $\tilde{X}^{n}$ satisfied the
$\partial\bar\partial$-lemma, Theorem~\ref{blowup-criterion} would imply that
$Z^{n-2}$ also satisfied it, a contradiction.  Hence $X^{n}$ is a
$\partial\bar\partial$-manifold while its modification $\tilde{X}^{n}$ is
not.
\end{proof}

The ambient manifold is genuinely non-K\"ahler and non-Moishezon.

\begin{cor}\label{nonkahler-ambient}
The threefold $\widehat Z$ in Theorem~\ref{main-thm} is not K\"ahler.  For
Honda's choice with algebraic dimension one, it is also not Moishezon.
\end{cor}

\begin{proof}
A closed complex submanifold of a compact K\"ahler manifold is K\"ahler and
therefore satisfies the $\partial\bar\partial$-lemma.  Since
$D\subset\widehat Z$ does not satisfy it, $\widehat Z$ cannot be K\"ahler.
Honda proves that the deformed twistor space $Z$ has algebraic dimension one
\cite[Proposition 3.1 and the discussion following it]{Hn15}.
Algebraic dimension is bimeromorphically invariant, so
\[
 a(\widehat Z)=a(Z)=1<3=\dim_{\mb C}\widehat Z.
\]
Thus $\widehat Z$ is not Moishezon.
\end{proof}

\begin{rem}\label{codimension-remark}
The surface $D$ has codimension one in $\widehat Z$.  This is compatible with
the bimeromorphic invariance results for compact complex threefolds due to
S.~Yang--X.~Yang \cite[Theorem 1.3]{YY20} and Rao--Zou
\cite[Corollary 5.4]{RZ24}: blowing up a
smooth divisor is an isomorphism, whereas the blow-up criterion in
Theorem~\ref{blowup-criterion} concerns centers of codimension at least two.  The
example therefore separates the heredity problem for submanifolds from the
bimeromorphic-invariance problem in complex dimension three.
\end{rem}

The following lemma shows that if the threefold $\widehat Z$ in
Theorem~\ref{main-thm} admitted an SKT metric, every smooth compact complex
surface contained in it  would be K\"ahler, so the non-K\"ahler divisor $D$
obstructs the existence of such a metric on $\widehat Z$. 

\begin{lemma}\label{skt-surface-kahler}

Let $X$ be a compact complex threefold satisfying the
$\partial\bar\partial$-lemma.  Suppose that $X$ carries an \emph{SKT (or
pluriclosed) Hermitian form} $\omega$, that is,
\[
 \omega>0,
 \qquad
 \partial\bar\partial\omega=0.
\]
Then every compact complex surface $S\subset X$ is K\"ahler.
\end{lemma}

\begin{proof}
Assume for contradiction that $S$ is non-K\"ahler.  A.~Lamari's form of the
Buchdahl--Lamari criterion implies that $S$ admits a nonzero positive
de Rham-exact $(1,1)$-current $\Theta$; see \cite[Theorem 6.1]{La99} and
L.~Ornea--M.~Verbitsky--V.~Vuletescu \cite[Theorem 1.3]{OVV21}.  Thus there
is a current $R$ on $S$ such that
\[
 0\neq\Theta\geq0,
 \qquad
 \Theta=dR.
\]

Let $\iota:S\hookrightarrow X$ be the inclusion and put
\[
 T:=\iota_*\Theta.
\]
Since $S$ has complex codimension one, $T$ is a positive $(2,2)$-current on $X$ by \S\ 2.3.
Moreover, pushforward commutes with $d$, and hence
\[
 T=\iota_*(dR)=d(\iota_*R).
\]
Thus $T$ is de Rham-exact.  The $\partial\bar\partial$-lemma is equivalent
for smooth forms and currents, so there is a $(1,1)$-current $U$ on $X$ such
that
\[
 T=\sqrt{-1}\partial\bar\partial U.
\]

Pairing $T$ with the SKT form gives
\[
 \langle T,\omega\rangle
 =\langle\Theta,\iota^*\omega\rangle>0,
\]
as $\iota^*\omega$ is a strictly positive Hermitian $(1,1)$-form on
$S$ and $\Theta$ is a nonzero positive current.  On the other hand,
integration by parts for currents and the pluriclosed condition yield
\[
\begin{aligned}
 \langle T,\omega\rangle
 &=\langle \sqrt{-1}\partial\bar\partial U,\omega\rangle\\
 &=\langle U,\sqrt{-1}\partial\bar\partial\omega\rangle\\
 &=0.
\end{aligned}
\]
This contradiction proves that $S$ is K\"ahler.
\end{proof}

\begin{rem}\label{skt-surface-remark}
The proof shows that a non-K\"ahler surface would produce a positive exact
current whose pushforward is a $\partial\bar\partial$-boundary on $X$; the
SKT form would then assign this current both positive and zero mass.  The
conclusion is only that $S$ admits a K\"ahler metric; it does not assert that
the restricted Hermitian form $\iota^*\omega$ is itself K\"ahler.
\end{rem}

\begin{rem}\label{hermitian-symplectic-remark}
On a $\partial\bar\partial$-manifold, every SKT form is
\emph{Hermitian-symplectic}.  Indeed, $\partial\omega$ is $d$-closed and
$\partial$-exact, so the $\partial\bar\partial$-lemma gives a $(1,0)$-form
$\xi$ such that
\[
 \partial\omega=\partial\bar\partial\xi.
\]
Consequently,
\[
 \Omega:=\omega+\partial\xi+\bar\partial\bar\xi
\]
is a real closed $2$-form whose $(1,1)$-part is the positive form $\omega$.
Thus $\Omega$ is symplectic and tames the complex structure.  For every
complex submanifold $Y\subset X$, the restriction $\Omega|_Y$ is again a
symplectic form taming $J|_Y$; hence $Y$ is Hermitian-symplectic.

The expected higher-dimensional extension of
Lemma~\ref{skt-surface-kahler} is therefore that every compact complex submanifold
$Y\subset X$ is K\"ahler.  This is an instance of the conjecture introduced
by J. Streets--G. Tian \cite[Question~1.7]{ST10}, which asserts that every compact
Hermitian-symplectic complex manifold is K\"ahler.  The conjecture holds in
complex dimension two, but remains open in general in every complex dimension
at least three; see Y.~Guo--F.~Zheng
\cite[Conjecture 1 and \S~1]{GZ26}.  For another recent discussion, see
T.-J.~Li--S.~Ning \cite[\S~1.1]{LN26}.  Under the additional assumption
that $Y$ belongs to the class $\mc C$ of Fujiki, the restricted form
$\omega|_Y$ is SKT, and I.~Chiose \cite[Theorem 2.2]{Ch14} implies that $Y$
is K\"ahler.
\end{rem}

\section{Derived blow-up formulae for generalized Bott--Chern complexes}
\label{derived-part}

In \S~\ref{threefold-part}, the classical blow-up criterion plays two
roles: it preserves the $\partial\bar\partial$-lemma along the embedded
resolution, and it detects its failure after blowing up a higher-dimensional
manifold along a center that does not satisfy the lemma.  We now turn from
these geometric applications to the sheaf-theoretic and derived structure
of blow-up formulae.  Rather than using only decompositions of cohomology
groups, we construct decompositions at the level of complexes of sheaves,
allowing general coefficients and non-compact complex manifolds.

We first establish local projective-bundle and blow-up formulae for
smooth-form double complexes, and then prove the derived Bott--Chern formula
with general coefficients by means of cup products and Gysin morphisms.

\subsection{Blow-up formula for smooth-form double complexes}

We follow the double-complex approach of Stelzig
\cite[\S\S~3 and~5]{Se21}.
The aim of this subsection is to lift the projective-bundle and blow-up
formulae for smooth-form double complexes from global cohomology to the
sheaf level.  Using currents, pushforward, and restriction to the exceptional
divisor, we obtain local $E_1$-isomorphisms without imposing compactness on
the ambient manifold; compare \cite[\S~5]{Se21}.

Let $X$ be a complex manifold and $V$ a rank-$c$ holomorphic vector bundle on
$X$, with associated projective bundle
$\varpi:\mathbb P(V)\longrightarrow X$.  Denote by
$\mathcal O_{\mathbb P(V)}(1)$ the dual of the tautological bundle, and
choose a closed $(1,1)$-form $\theta$ representing its first Chern class.
These data define a morphism
\begin{equation}\label{derived-eq-4-1}
u:\bigoplus_{r=0}^{c-1}\mathcal E_X[-r]
\longrightarrow \varpi_*\mathcal E_{\mathbb P(V)},
\qquad
(\alpha_r)_{r=0}^{c-1}\longmapsto
\sum_{r=0}^{c-1}\varpi^*\alpha_r\wedge\theta^{\,r}.
\end{equation}
Here $\mathcal E_X=(\mathcal E_X^{s,t},\partial,\bar\partial)$ denotes
the double complex of sheaves of smooth complex-valued differential forms
on $X$, with $\mathcal E_X^{s,t}$ the sheaf of smooth $(s,t)$-forms;
$\mathcal E_{\mathbb P(V)}$ is defined similarly.
\begin{Pro}\label{prop-smooth-projective-bundle}
The morphism $u$ defined by \eqref{derived-eq-4-1} is an $E_1$-isomorphism.
\end{Pro}
\begin{proof}
It suffices to show that for every open subset $U\subset X$,
\[u\vert_U: \bigoplus_{r=0}^{c-1} \mathcal E_X(U)[-r]
\longrightarrow \mathcal E_{\mathbb{P}(V)}(\varpi^{-1}(U)) \]
is an $E_1$-isomorphism.  By \cite[\S~3]{Se21} and \cite[\S~3]{RYY20}, this follows from the
Dolbeault Hirsch lemma and Borel's spectral sequence.
\end{proof}

Consider a proper holomorphic map $f:Y\longrightarrow X$ between complex
manifolds of the same dimension $n$.  Write $\mathcal D\mathcal E_X$ for
the double complex whose component in bidegree $(s,t)$ is the sheaf of germs
of $(s,t)$-currents on $X$.  Thus, for an open subset $U\subset X$,
$\mathcal D\mathcal E_X^{s,t}(U)$ is the dual of the space of compactly
supported $(n-s,n-t)$-forms on $U$.  Integration gives a natural morphism
\[ \Phi_Y: \mathcal E_Y \longrightarrow \mathcal{D}\mathcal E_Y, \quad
\omega \longmapsto \left( \varphi \mapsto \int_Y \omega \wedge \varphi \right).\]
By \cite[Theorem~1]{Serre}, this is an $E_1$-isomorphism.  Since $f$ is proper, the
pullback $f^*:\mathcal E_X\longrightarrow f_*\mathcal E_Y$ induces a
morphism
\[f_\#: f_*\mathcal{D}\mathcal E_Y
\longrightarrow \mathcal{D}\mathcal E_X.\]
Define the pushforward morphism by
\[f_*=f_\#\circ f_*(\Phi_Y):f_*\mathcal E_Y
\longrightarrow\mathcal D\mathcal E_X.\]
\begin{Le}\label{prop-smooth-pushforward}
Suppose moreover that $f$ is surjective.  Let
$q:f_*\mathcal E_Y\longrightarrow
f_*\mathcal E_Y/\operatorname{im}(f^*)$
be the quotient map.  Then the morphism
\[(f_*, q):
f_*\mathcal E_Y \longrightarrow \mathcal{D}\mathcal E_X \oplus f_*\mathcal E_Y / \operatorname{im}(f^{*})
\]
 is an $E_1$-isomorphism.
\end{Le}
\begin{proof}
We first show that $f^*:\mathcal E_X\longrightarrow f_*\mathcal E_Y$ is
injective.  By \cite{Gu}, the restriction
$f|_{f^{-1}(U)}:f^{-1}(U)\longrightarrow U$ is a finite covering over a
dense open subset $W\subset U$.  The claim follows by restricting to local
biholomorphic branches over $W$, and hence we have an exact sequence
\[
0 \longrightarrow \mathcal E_X \xrightarrow{f^{*}} f_*\mathcal E_Y
\xrightarrow{q}f_*\mathcal E_Y / \operatorname{im}(f^{*}) \longrightarrow 0,
\]
which yields a long exact sequence:
\begin{equation}\label{derived-eq-4-2}
\cdots \longrightarrow H^{p,q}_{\partial_1}(\mathcal E_X) \xrightarrow{f^{*}} H^{p,q}_{\partial_1}(f_*\mathcal E_Y)
\xrightarrow{q} H^{p,q}_{\partial_1}\bigl(f_*\mathcal E_Y / \operatorname{im}(f^{*})\bigr)
\xrightarrow{\delta} H^{p+1,q}_{\partial_1}(\mathcal E_X) \longrightarrow \cdots.
\end{equation}
By \cite[Lemma~2.1]{W}, the following diagram commutes up to multiplication by
$\deg f>0$:
\[
\begin{adjustbox}{max width=\linewidth}
\small$
\begin{tikzcd}[ampersand replacement=\&, column sep=1.8em, row sep=1.8em]
\mathcal E_X \ar[r, "\Phi_X"] \ar[d, "f^*"] \& \mathcal{D}\mathcal E_X \\
f_*\mathcal E_Y \ar[r, "f_*(\Phi_Y)"] \& f_*\mathcal{D}\mathcal E_Y.\ar[u, "f_\#"]
\end{tikzcd}
$
\end{adjustbox}
\]
This shows that, for any $p,q\in\mathbb Z$, the induced morphism
\[
f_*f^*=\deg(f)\,\Phi_X:
H_{\partial_1}^{p,q}(\mathcal E_X)
\longrightarrow H_{\partial_1}^{p,q}(\mathcal{D}\mathcal E_X)
\]
is an isomorphism.
Hence \eqref{derived-eq-4-2} reduces to a split exact sequence:
\[
0 \longrightarrow H^{p,q}_{\partial_1}(\mathcal{D}\mathcal E_X) \xrightarrow{f^{*}\circ\Phi_X^{-1}} H^{p,q}_{\partial_1}(f_*\mathcal E_Y)
\xrightarrow{q} H^{p,q}_{\partial_1}\bigl(f_*\mathcal E_Y / \operatorname{im}(f^{*})\bigr)
\longrightarrow 0.
\]
 This completes the proof.
\end{proof}

Now fix a closed complex submanifold $\iota:Z\hookrightarrow X$ of codimension
$c$, and denote the blow-up along $Z$ by
$\pi:\tilde{X}\longrightarrow X$.  The exceptional divisor is the
projectivization $E=\mathbb P(N_{Z/X})$ \cite[\S~3.3.3]{Vi02}.  We write
$\jmath:E\hookrightarrow\tilde{X}$ for the inclusion and
$w=\pi|_{E}:E\longrightarrow Z$ for the bundle projection
\begin{equation}\label{derived-eq-4-3}
\begin{adjustbox}{max width=\linewidth}
\small$
\begin{tikzcd}[ampersand replacement=\&, column sep=1.8em, row sep=1.8em]
	E \arrow[r, "\jmath"] \arrow[d, "w"'] \& \tilde{X} \arrow[d, "\pi"] \\
	Z \arrow[r, "\iota"'] \& X.
\end{tikzcd}
$
\end{adjustbox}
\end{equation}

\begin{Le}\label{prop-smooth-blowup-cokernel}
Let
\[
\pi^*: \mathcal E_X \longrightarrow \pi_*\mathcal E_{\tilde{X}}
\quad\text{and}\quad
w^*: \mathcal E_Z \longrightarrow w_*\mathcal E_{E}
\]
be the natural morphisms induced by pullback of differential forms.  Then
\[
h: \operatorname{coker}(\pi^*)
\longrightarrow \iota_*\operatorname{coker}(w^*)
\]
induced by $\jmath^*$ is an $E_1$-isomorphism.

\end{Le}

\begin{proof}
Fix an open subset $U\subset X$.  Since every component of $\mathcal E_U$
is a soft sheaf, we have
\[\operatorname{coker}(\pi^*)(U)=\mathcal E_{\pi^{-1}(U)}(\pi^{-1}(U))/\pi^* \mathcal E_{U}(U).\]
Similarly,  \[\iota_*\operatorname{coker}(w^*)(U)=\mathcal E_{w^{-1}(U\cap Z)}(w^{-1}(U\cap Z))/w^* \mathcal E_{U\cap Z}(U\cap Z). \]
If $U\cap Z=\varnothing$, both cokernels vanish, so $h_U$ is clearly an
$E_1$-isomorphism.  Otherwise, write
$g:\pi^{-1}(U)\longrightarrow U$ for the restriction of $\pi$.  This map is
the blow-up of $U$ along $Z\cap U$.  The following isomorphisms originate
in the algebraic setting in \cite[chapitre~IV, th\'eor\`eme~1.2.1]{G}
and are established explicitly in the proof of \cite[proposition~(3.3)]{GNA}.
For the version used here, valid for possibly non-compact complex
manifolds, see \cite[Lemma~4.1]{M}:
\begin{align*}
g^*: \Omega_U^p
&\xrightarrow{\sim} g_*\Omega_{\pi^{-1}(U)}^p,\\
\bar{w}^*: \Omega_{Z\cap U}^p
&\xrightarrow{\sim} \bar{w}_*\Omega_{E\cap \pi^{-1}(U)}^p,\\
\jmath^*: R^sg_*\Omega_{\pi^{-1}(U)}^p
&\xrightarrow{\sim} \iota_*R^s\bar{w}_*\Omega_{E\cap \pi^{-1}(U)}^p,
\qquad s\geq1,
\end{align*}
where $p,s\in\mathbb N$ and
$\bar{w}:E\cap\pi^{-1}(U)\longrightarrow Z\cap U$ is the restriction of
$w$.  The proof of \cite[Theorem 5.1]{Se21} then applies verbatim and shows
that $h_U$ is an $E_1$-isomorphism.
\end{proof}
Note that $\pi$ is a surjective proper holomorphic map between complex
manifolds of the same dimension and that $w$ is a projective bundle of rank
$c-1$.
Combining Proposition~\ref{prop-smooth-projective-bundle} with
Lemma~\ref{prop-smooth-pushforward} and
Lemma~\ref{prop-smooth-blowup-cokernel}, we
obtain the following result:
\begin{Pro}\label{prop-smooth-blowup-zigzag}
There is a zigzag of $E_1$-isomorphisms:
\[
\begin{adjustbox}{max width=\linewidth}
\small$
\begin{tikzcd}[ampersand replacement=\&, column sep=1.8em, row sep=1.8em]
\pi_*\mathcal E_{\tilde{X}} \arrow[r, "{(\pi_*,h\circ q_X)}"] \& \mathcal{D}\mathcal E_X\oplus \iota_*\operatorname{coker}(w^*) \& \\
\& \mathcal{D}\mathcal E_X\oplus \bigoplus_{r=1}^{c-1}\iota_*\mathcal E_Z[-r] \arrow[u, "{\id\oplus\iota_*\bar u}"] \&  \mathcal E_X\oplus \bigoplus_{r=1}^{c-1}\iota_*\mathcal E_Z[-r].\arrow[l, "{\Phi_X\oplus\id}"']
\end{tikzcd}
$
\end{adjustbox}
\]
Here $q_X:\pi_*\mathcal E_{\tilde X}\to\operatorname{coker}(\pi^*)$
is the quotient map, and
\[
\bar u:\bigoplus_{r=1}^{c-1}\mathcal E_Z[-r]\longrightarrow\operatorname{coker}(w^*),
\qquad
(\alpha_r)_{r=1}^{c-1}\longmapsto
\left[\sum_{r=1}^{c-1}w^*\alpha_r\wedge\theta_E^r\right],
\]
where $\theta_E$ is a closed $(1,1)$-form representing $c_1(\mathcal O_E(1))$.
\end{Pro}
\begin{Rem}\label{rem-local-sections-e1}\normalfont
Suppose that $F:K\longrightarrow L$ is a quasi-isomorphism between bounded
below complexes of sheaves on $X$ and that, for every $i\in\mathbb Z$,
$K^i$ and $L^i$ are acyclic for the section functor over every open set.
Then $F_U:K(U)\longrightarrow L(U)$ is a quasi-isomorphism for every open
set $U\subset X$.  Consequently,
Proposition~\ref{prop-smooth-blowup-zigzag} also holds for the
$E_1$-isomorphisms induced on local sections.
\end{Rem}

\begin{Cor}\label{cor-smooth-truncated-blowup}
The $E_1$-isomorphisms in Proposition~\ref{prop-smooth-blowup-zigzag}
induce isomorphisms in $D^b(\mathbb C_X)$:
\begin{equation}\label{derived-eq-4-4}\eta:\Omega_X^{[p,q]}\oplus\bigoplus_{r=1}^{c-1}\iota_* \Omega_Z^{[p-r,q-r]}[-2r]\xrightarrow{\sim}R\pi_*\Omega_{\tilde{X}}^{[p,q]},\end{equation}
and
\begin{equation}\label{derived-eq-4-5}\mu:\mathbb C_X\oplus\bigoplus_{r=1}^{c-1}\iota_*\mathbb C_Z[-2r]\xrightarrow{\sim}R\pi_*\mathbb C_{\tilde{X}}.\end{equation}
These isomorphisms are compatible with holomorphic truncation. More precisely,
for a complex manifold $W$, let
$\tau_W^a:\mathbb C_W\to\Omega_W^{[0,a]}$ be the inclusion of locally
constant functions in degree zero, with $\tau_W^a=0$ when $a<0$.
For every $p\in\mathbb Z$, the diagram
\[
\begin{adjustbox}{max width=\linewidth}
\small$
\begin{tikzcd}[ampersand replacement=\&, column sep=1.8em, row sep=1.8em]
R\pi_*\Omega_{\tilde{X}}^{[0,p-1]}
\& \Omega_X^{[0,p-1]}\oplus
 \displaystyle\bigoplus_{r=1}^{c-1}\iota_*\Omega_Z^{[0,p-1-r]}[-2r]
 \arrow[l,"\eta"] \\
R\pi_*\mathbb C_{\tilde{X}}
 \arrow[u,"R\pi_*\tau_{\tilde{X}}^{p-1}"]
\& \mathbb C_X\oplus
 \displaystyle\bigoplus_{r=1}^{c-1}\iota_*\mathbb C_Z[-2r]
 \arrow[l,"\mu"]
 \arrow[u,"{\tau_X^{p-1}\oplus\bigoplus_r \iota_*\tau_Z^{p-1-r}[-2r]}"']
\end{tikzcd}
$
\end{adjustbox}
\]
commutes in $D^b(\mathbb C_X)$, where $\eta$ is specialized to the
truncation interval $[0,p-1]$.
\end{Cor}
\begin{proof}
Applying $T^{p,q}$ to the zigzag in
Proposition~\ref{prop-smooth-blowup-zigzag} gives a zigzag of
quasi-isomorphisms by Remark~\ref{rem-local-sections-e1},
Proposition~\ref{prop-double-complex-functor}, and
Example~\ref{ex-Tpq-functor}. The Dolbeault resolutions and the identity
\[
T^{p,q}(\mathcal E_Z[-r])
=T^{p-r,q-r}(\mathcal E_Z)[-2r]
\]
give \eqref{derived-eq-4-4}. Here the terms of the smooth-form resolutions
on $\tilde X$ are $\pi_*$-acyclic, so their direct images compute
the derived direct images. Taking the full total complexes and using
the de Rham resolutions gives \eqref{derived-eq-4-5}.

Put $n=\dim_{\mathbb C}X$. The natural truncation projections in
Proposition~\ref{prop-smooth-blowup-zigzag} give the commutative diagram
\[
\begin{adjustbox}{max width=\linewidth}
\small$
\begin{tikzcd}[ampersand replacement=\&, row sep=0.8em, column sep=0.8em]
\& \pi_*T^{0,n}\mathcal E_{\tilde{X}} \arrow[r] \arrow[ddl,"p"]
\& T^{0,n}\mathcal D\mathcal E_X\oplus  \iota_*T^{0,n}\operatorname{coker}(w^*) \arrow[ddl,"p"] \& \\
\& \& T^{0,n}\mathcal D\mathcal E_X\oplus \bigoplus_{r=1}^{c-1}\iota_*T^{0,n}(\mathcal E_Z[-r]) \arrow[u] \arrow[ddl,"p"]
\& T^{0,n}\mathcal E_X\oplus \bigoplus_{r=1}^{c-1}\iota_*T^{0,n}(\mathcal E_Z[-r]) \arrow[l] \arrow[ddl,"p"] \\
\pi_*T^{0,p-1}\mathcal E_{\tilde{X}} \arrow[r]
\& T^{0,p-1}\mathcal D\mathcal E_X\oplus  \iota_*T^{0,p-1}\operatorname{coker}(w^*) \& \& \\
\& T^{0,p-1}\mathcal D\mathcal E_X\oplus \bigoplus_{r=1}^{c-1}\iota_*T^{0,p-1}(\mathcal E_Z[-r]) \arrow[u]
\& T^{0,p-1}\mathcal E_X\oplus \bigoplus_{r=1}^{c-1}\iota_*T^{0,p-1}(\mathcal E_Z[-r]). \arrow[l] \&
\end{tikzcd}
$
\end{adjustbox}
\]
Here, by abuse of notation, $p$ denotes all four truncation projections.
For $1\le r\le c-1$,
\[
T^{0,n}(\mathcal E_Z[-r])
=T^{-r,n-r}(\mathcal E_Z)[-2r]
=\mathcal E_Z^\bullet[-2r],
\]
where $\mathcal E_Z^\bullet$ denotes the de Rham complex of sheaves of
smooth complex-valued differential forms on $Z$,
whereas
\[
T^{0,p-1}(\mathcal E_Z[-r])
=T^{-r,p-1-r}(\mathcal E_Z)[-2r]
\simeq\Omega_Z^{[0,p-1-r]}[-2r].
\]
Under the de Rham and Dolbeault resolutions, the truncation projections
correspond to $\tau_X^{p-1}$, $\tau_{\tilde X}^{p-1}$, and
$\tau_Z^{p-1-r}[-2r]$ on the indicated summands. Passing to the
derived category therefore gives the asserted compatibility diagram.
\end{proof}

\begin{Rem}\label{rem-derived-dolbeault-coefficients}\normalfont
Let $M\in D^b(\mathcal O_X)$.  For $p,q\in\mathbb Z$, define
\[
H^{p,q}(X,M):=\mathbb H^q\bigl(X,\Omega_X^p\otimes_{\mathcal O_X}M\bigr),
\]
where $\mathbb H^q$ denotes hypercohomology.  Since $\Omega_X^p$ is locally
free, the ordinary tensor product here also computes the derived tensor
product.  Setting $p=q$ in \eqref{derived-eq-4-4}, we obtain
\[
R\pi_*\Omega_{\tilde{X}}^p
\cong\Omega_X^p\oplus
\bigoplus_{r=1}^{c-1}\iota_*\Omega_Z^{p-r}[-r]
\]
in $D^b(\mathcal O_X)$.  Tensoring with $M$ and applying
Corollary~\ref{cor-derived-projection-formula} to $\pi$ and $\iota$, we obtain
an isomorphism in $D^b(\mathcal O_X)$:
\[
\begin{aligned}
R\pi_*\bigl(\Omega_{\tilde{X}}^p\otimes_{\mathcal O_{\tilde{X}}}L\pi^*M\bigr)
&\cong (R\pi_*\Omega_{\tilde{X}}^p)\otimes^L_{\mathcal O_X}M\\
&\cong (\Omega_X^p\otimes_{\mathcal O_X}M)\oplus
\bigoplus_{r=1}^{c-1}
\iota_*\bigl(\Omega_Z^{p-r}\otimes_{\mathcal O_Z}L\iota^*M\bigr)[-r].
\end{aligned}
\]
Taking hypercohomology on $X$ yields
\[
H^{p,q}(\tilde{X},L\pi^*M)
\cong H^{p,q}(X,M)\oplus
\bigoplus_{r=1}^{c-1}H^{p-r,q-r}(Z,L\iota^*M).
\]
When $M$ is a
holomorphic vector bundle placed in degree zero, the derived pullbacks
agree with the ordinary pullbacks.  Thus, for compact $X$, this recovers
the bundle-valued Dolbeault blow-up formula of Rao--S.~Yang--X.~Yang
\cite[Theorem 1.2]{RYY20}.  Without assuming compactness, it recovers
Meng's extension \cite[Theorem 6.6]{Mn19MV}.  The formula above further
extends these results to bounded derived coefficients.
The classical formula with ordinary pullbacks does not extend in general
to coherent sheaves.  The formulae \eqref{derived-eq-4-4} and
\eqref{derived-eq-4-5} do not extend unchanged to singular ambient spaces;
see \S~\ref{subsec-failure-locally-free-coefficients} for counterexamples.
\end{Rem}

\subsection{Derived blow-up formula for Bott--Chern complexes}

This final subsection of the derived part constructs the Bott--Chern
blow-up morphism with coefficients in an arbitrary subring of $\mathbb C$.
The construction combines the constant-sheaf projective-bundle
decomposition, cup products, and Gysin morphisms, and its hypercohomology
recovers the cohomological blow-up formula.

By Example~\ref{ex-Lpq-functor}, $H^n(L^{p,q})$ is a linear functor that
maps every square to $0$.  Therefore, the derived Bott--Chern blow-up
formula with complex coefficients follows from the functorial criterion
for bounded double complexes and the smooth-form blow-up decomposition above.
The proof of the blow-up formula for the general case is more subtle.

Fix a commutative ring $S$ with unit, and retain the notation of
diagram~\eqref{derived-eq-4-3}.
By \cite[\S~2.6]{KS}, for any $\mathcal{F}\in D^b(S_Z)$ and
$\mathcal{G}\in D^b(S_E)$, there is a natural isomorphism
\[
\mathrm{Hom}_{D^b(S_Z)}(\mathcal F,Rw_*\mathcal G)
\cong
\mathrm{Hom}_{D^b(S_E)}(w^{-1}\mathcal F,\mathcal G).
\]
Taking $\mathcal F=S_Z[-2k]$ and $\mathcal G=S_{E}$, we obtain
   \begin{equation}\label{derived-eq-6-1}\mathrm{Hom}_{D^b(S_Z)}(S_Z[-2k], Rw_*S_{E})\cong \mathrm{Hom}_{D^b(S_E)}(S_{E}[-2k], S_{E})\cong H^{2k}(E;S).\end{equation}
Denote by $\theta_E$ the image of $c_1(\mathcal O_{E}(1))$ under the natural
morphism
\[
H^2(E;\mathbb Z)\longrightarrow H^2(E;S).
\]
Write $\alpha_0$ for the adjunction morphism
\[
S_Z\longrightarrow Rw_*w^{-1}S_Z=Rw_*S_{E}.
\]
For each integer $1\le k\le c-1$, let
$\alpha_k\in\mathrm{Hom}_{D^b(S_Z)}(S_Z[-2k],Rw_*S_{E})$ be the morphism
whose image is $\theta_E^k$ under \eqref{derived-eq-6-1}.  Define
\[\alpha=(\alpha_0,\alpha_1,\ldots,\alpha_{c-1}):
\bigoplus_{k=0}^{c-1}S_Z[-2k]\longrightarrow Rw_*S_{E}.\]
Then $\alpha$ is a quasi-isomorphism since, for any $z\in Z$, the map
$(\alpha_k)_z$ sends the unit of $S$ to the generator
\[
\theta_E^k\vert_{w^{-1}(z)}\in H^{2k}(\mathbb P^{c-1};S)
\cong (R^{2k}w_*S_{E})_z.
\]

By \cite[\S~3.3]{KS}, the Thom isomorphism gives
$S_{E}[-2]\cong \jmath^!S_{\tilde{X}}$.  The
\emph{Gysin map} is therefore the composition
  \[\mathrm{Gy}_{E}: \jmath_*S_{E}[-2] \cong \jmath_*\jmath^!S_{\tilde{X}}\cong R\jmath_!\jmath^!S_{\tilde{X}}\longrightarrow  S_{\tilde{X}}.\]
Here the last arrow is the adjunction morphism.
Since $\iota\circ w=\pi\circ \jmath$, we have
$\iota_*Rw_*=R\pi_*\jmath_*$.  Set
\[\gamma:\iota_*Rw_*S_{E}[-2]=R\pi_*\jmath_*S_{E}[-2]
\xrightarrow{R\pi_*(\mathrm{Gy}_{E})}R\pi_*S_{\tilde{X}}.\]
For each $1\le k\le c-1$, let
\[
\lambda_k:\iota_*S_Z[-2k]\hookrightarrow
\bigoplus_{m=1}^{c}\iota_*S_Z[-2m]
\xrightarrow{-(\iota_*\alpha)[-2]}\iota_*Rw_*S_{E}[-2].
\]
Define
\begin{equation}\label{derived-eq-6-2}
\varphi^S=(\varphi_0^S,\varphi_1^S,\ldots,\varphi_{c-1}^S):
S_X\oplus\bigoplus_{k=1}^{c-1}\iota_*S_Z[-2k]
\longrightarrow R\pi_*S_{\tilde{X}},
\end{equation}
where $\varphi_0^S$ is the adjunction morphism and
$\varphi_k^S=\gamma\circ\lambda_k$.

For integral coefficients, the derived constant-sheaf blow-up
decomposition is stated in \cite[\S~2.4.4]{D}. Using the explicit
morphism \eqref{derived-eq-6-2}, we prove the decomposition below
for an arbitrary coefficient ring.

\begin{Pro}\label{prop-constant-sheaf-blowup}
The map $\varphi^S$ is an isomorphism in $D^b(S_X)$.
\end{Pro}
\begin{proof}
If $x\notin Z$, then $\varphi^S_x$ is an isomorphism, since $\pi$ is
biholomorphic over $X\setminus Z$.  If $x\in Z$, proper base change gives
\[
(R^m\pi_*S_{\tilde{X}})_x
\cong H^m(w^{-1}(x);S)
\cong H^m(\mathbb P^{c-1};S).
\]
In degree zero, $\mathcal H^0(\varphi^S)_x$ sends $1$ to $1$.
Write $N=N_{E/\tilde{X}}$. The self-intersection
formula gives\footnote{For real or complex coefficients, under the
de Rham comparison, the map induced by $\mathrm{Gy}_E$ agrees with the
Bott--Tu Thom--Gysin map for the complex orientation of $N_{E/\tilde{X}}$;
see \cite[\S~6]{BT82} for the Thom construction. Indeed, under the
de Rham comparison, the purity isomorphism is realized by the Thom
isomorphism, and the adjunction counit induces the forget-support map.}
\begin{equation}\label{eq-gysin-self-intersection}
\jmath^*\jmath_{\mathrm{Gys}}(1)=e(N_{\mathbb R})=c_1(N),
\qquad
\jmath_{\mathrm{Gys}}=\mathbb H^{k+2}(\tilde{X},\mathrm{Gy}_E).
\end{equation}
Since $\theta_E=\jmath^*c_1(\mathcal O_{\tilde{X}}(-E))$, the projection
formula and $c_1(N)=-\theta_E$ give
$\jmath^*\jmath_{\mathrm{Gys}}(\theta_E^{\ell-1})=-\theta_E^\ell$.
Together with the minus sign in $\lambda_\ell$, this shows that,
for $1\le\ell\le c-1$,
\[
\mathcal H^{2\ell}(\varphi^S)_x:
S\longrightarrow H^{2\ell}(\mathbb P^{c-1};S),
\qquad 1\longmapsto\theta_E^\ell|_{w^{-1}(x)}.
\]
This is an isomorphism, since the image is an $S$-module generator.
In all other degrees, both stalks vanish. Thus $\varphi^S$ induces
isomorphisms on all cohomology sheaves and hence is an isomorphism
in $D^b(S_X)$.
\end{proof}

The following lemma is straightforward:
\begin{Le}\label{lem-constant-sheaf-coefficient-change}
Let $f:T\longrightarrow S$ be a map of commutative rings with unit.
Then the following diagram commutes:
\[
\begin{adjustbox}{max width=\linewidth}
\small$
\begin{tikzcd}[ampersand replacement=\&, column sep=1.8em, row sep=1.8em]
T_X \oplus \displaystyle\bigoplus_{r=1}^{c-1} \iota_* T_Z[-2r]
  \arrow[r, "\varphi^T"]
  \arrow[d, "f_*"']
\& R\pi_* T_{\tilde X}
  \arrow[d, "R\pi_*f_*"]
\\
S_X \oplus \displaystyle\bigoplus_{r=1}^{c-1} \iota_* S_Z[-2r]
  \arrow[r, "\varphi^S"']
\& R\pi_* S_{\tilde X}.
\end{tikzcd}
$
\end{adjustbox}
\]
\end{Le}

\begin{Th}\label{thm-generalized-bc-blowup}
Let $X$ be a complex manifold, $\iota:Z\hookrightarrow X$ a
closed complex submanifold of codimension $c\geq2$, and
$\pi:\tilde{X}\longrightarrow X$ the blow-up along $Z$.
Let $S$ be a commutative ring with unit, and
$f,g:S\longrightarrow\mathbb C$ ring morphisms.
For every $p,q\in\mathbb Z$, there is an isomorphism
\begin{equation}\label{eq-derived-generalized-bc-blowup}
\mathcal B_X^{p,q}(f,g)\oplus
\bigoplus_{r=1}^{c-1}\iota_*\mathcal B_Z^{p-r,q-r}(f,g)[-2r]
\xrightarrow[\sim]{b}R\pi_*\mathcal B_{\tilde{X}}^{p,q}(f,g)
\end{equation}
in $D^b(\mathbb Z_X)$, compatible with the constant-sheaf isomorphism
$\varphi^S$ in \eqref{derived-eq-6-2}, i.e., the diagram
\[
\begin{adjustbox}{max width=\linewidth}
\small$
\begin{tikzcd}[ampersand replacement=\&, column sep=1.8em, row sep=1.8em]
\mathcal B_X^{p,q}(f,g)\oplus
 \displaystyle\bigoplus_{r=1}^{c-1}\iota_*\mathcal B_Z^{p-r,q-r}(f,g)[-2r]
 \arrow[r] \arrow[d,"b"']
\& S_X\oplus\displaystyle\bigoplus_{r=1}^{c-1}\iota_*S_Z[-2r]
 \arrow[d,"\varphi^S"] \\
R\pi_*\mathcal B_{\tilde{X}}^{p,q}(f,g) \arrow[r]
\& R\pi_*S_{\tilde{X}}
\end{tikzcd}
$
\end{adjustbox}
\]
commutes in $D^b(\mathbb Z_X)$.
Here $\mathcal B_X^{p,q}(f,g)$ is the generalized Bott--Chern complex
of type $(p,q)$ associated with $(f,g)$ in
Definition~\ref{def-generalized-bc-complex}.
The horizontal arrows are induced by the natural projections
$\mathcal B_Y^{p,q}(f,g)\longrightarrow S_Y$, with the indicated
shifts and direct sums on the top row and derived direct image on
the bottom row.
\end{Th}
\begin{proof}
Let $E$ be the exceptional divisor,
$\jmath:E\hookrightarrow\tilde{X}$ its inclusion, and
$\varpi:E\longrightarrow Z$ the projective-bundle projection.  Then
$\pi\circ\jmath=\iota\circ\varpi$.  By
Corollary~\ref{cor-smooth-truncated-blowup}, the isomorphisms $\eta$ and
$\mu$ are compatible with the holomorphic truncation maps $\tau_W^a$.

\smallskip
\phantomsection\label{step-generalized-bc-1}
\noindent\emph{Step 1: the de Rham zigzag.}
Choose a real closed $(1,1)$-form representing
$c_1(\mathcal O_E(1))$ and, by abuse of notation, also denote it by
$\theta_E$. Write
$j_W:\mathbb C_W\to\mathcal E_W^\bullet$ for the de Rham augmentation
on $W=X,Z,E,\tilde{X}$, and let
$\Phi_W:\mathcal E_W^\bullet\to\mathcal D\mathcal E_W^\bullet$
be the natural inclusion of forms into currents. Here $\mathcal E_W^\bullet$
is the de Rham complex of sheaves of smooth complex-valued differential
forms on $W$, and $\mathcal D\mathcal E_W^\bullet$ is the  complex
of sheaves of complex-valued currents on $W$.
Set
\[
\begin{aligned}
\mathcal Q_X^\bullet&=
 \operatorname{coker}(\pi^*:\mathcal E_X^\bullet
 \to\pi_*\mathcal E_{\tilde{X}}^\bullet),\\
\mathcal Q_Z^\bullet&=
 \operatorname{coker}(\varpi^*:\mathcal E_Z^\bullet
 \to \varpi_*\mathcal E_E^\bullet),
\end{aligned}
\]
Let
\[
\begin{aligned}
q_X:\pi_*\mathcal E_{\tilde X}^\bullet
&\longrightarrow\mathcal Q_X^\bullet,\\
q_Z:\varpi_*\mathcal E_E^\bullet
&\longrightarrow\mathcal Q_Z^\bullet
\end{aligned}
\]
be the quotient maps. Restriction induces a
quasi-isomorphism $h:\mathcal Q_X^\bullet\to \iota_*\mathcal Q_Z^\bullet$
by Lemma~\ref{prop-smooth-blowup-cokernel}. For $0\le r\le c-1$,
define
\[
u_r:\mathcal E_Z^\bullet[-2r]\longrightarrow \varpi_*\mathcal E_E^\bullet,
\qquad \alpha\longmapsto \varpi^*\alpha\wedge\theta_E^r,
\]
For $r\ge1$, put
\[
v_r=\iota_*q_Z\circ\iota_*u_r:
\iota_*\mathcal E_Z^\bullet[-2r]\longrightarrow
\iota_*\mathcal Q_Z^\bullet.
\]
Let
\[
\pi_*:\pi_*\mathcal E_{\tilde X}^\bullet
\longrightarrow\mathcal D\mathcal E_X^\bullet
\]
denote the \emph{pushforward of smooth forms}: for an open
set $U\subset X$ and $\alpha\in\mathcal E_{\tilde X}^k(\pi^{-1}(U))$,
\[
(\pi_*\alpha)(\beta)
=\int_{\pi^{-1}(U)}\alpha\wedge \pi^*\beta,
\]
where $\beta$ is a compactly supported smooth form of degree $2n-k$
on $U$, where $n=\dim_{\mathbb C}X$. Properness of $\pi$ ensures that
$\pi^*\beta$ has compact support.
The full de Rham zigzag is represented by the quasi-isomorphisms
\[
\begin{aligned}
\Psi=(\pi_*,h\circ q_X):\pi_*\mathcal E_{\tilde{X}}^\bullet
&\longrightarrow\mathcal D\mathcal E_X^\bullet\oplus \iota_*\mathcal Q_Z^\bullet,\\
\Xi=\Phi_X\oplus(v_1,\ldots,v_{c-1}):
\mathcal E_X^\bullet\oplus\bigoplus_{r=1}^{c-1}\iota_*\mathcal E_Z^\bullet[-2r]
&\longrightarrow\mathcal D\mathcal E_X^\bullet\oplus \iota_*\mathcal Q_Z^\bullet.
\end{aligned}
\]
Let
\[
R\pi_*j_{\tilde X}:R\pi_*\mathbb C_{\tilde X}
\xrightarrow{\sim}R\pi_*\mathcal E_{\tilde X}^\bullet
\simeq\pi_*\mathcal E_{\tilde X}^\bullet
\]
be the de Rham comparison isomorphism. The definition
of $\mu$ from the zigzag gives
\[
(R\pi_*j_{\tilde X})\circ\mu
=\Psi^{-1}\circ\Xi\circ
\left(j_X\oplus\bigoplus_{r=1}^{c-1}(\iota_*j_Z)[-2r]\right).
\]
Write $\mu_0$ for the restriction of $\mu$ to $\mathbb C_X$ and,
for $1\le r\le c-1$, write $\mu_r$ for its restriction to
$\iota_*\mathbb C_Z[-2r]$.

\smallskip
\phantomsection\label{step-generalized-bc-2}
\noindent\emph{Step 2: the ambient component.}
For $\zeta\in\mathcal E_X^k(U)$ and a compactly supported test form $\beta$
of complementary degree on $U$, the degree-one property of $\pi$ gives
\[
\int_{\pi^{-1}(U)}\pi^*\beta\wedge\pi^*\zeta
=\int_U\beta\wedge\zeta.
\]
Hence $\pi_*\circ\pi^*=\Phi_X$, while $q_X\circ\pi^*=0$.
Consequently,
\[
\begin{aligned}
\Psi\circ\pi^*\circ j_X
&=\bigl(\pi_*\circ\pi^*\circ j_X,
        h\circ q_X\circ\pi^*\circ j_X\bigr)\\
&=(\Phi_X\circ j_X,0).
\end{aligned}
\]
On the other hand, restricting the defining identity for $\mu$ to
the summand $\mathbb C_X$ gives
\[
\Psi\circ(R\pi_*j_{\tilde X})\circ\mu_0
=(\Phi_X\circ j_X,0)
=\Psi\circ\pi^*\circ j_X.
\]
Since $\Psi$ is an isomorphism, we obtain
\[
(R\pi_*j_{\tilde X})\circ\mu_0=\pi^*\circ j_X.
\]
The diagram
\[
\begin{adjustbox}{max width=\linewidth}
\small$
\begin{tikzcd}[ampersand replacement=\&, column sep=1.8em, row sep=1.8em]
\mathbb C_X \arrow[r,"j_X","\sim"'] \arrow[d,"\varphi_0^{\mathbb C}"']
\& \mathcal E_X^\bullet \arrow[d,"\pi^*"] \\
R\pi_*\mathbb C_{\tilde{X}} \arrow[r,"R\pi_*j_{\tilde X}","\sim"']
\& \pi_*\mathcal E_{\tilde{X}}^\bullet
\end{tikzcd}
$
\end{adjustbox}
\]
commutes by naturality of the adjunction unit. Hence
\[
(R\pi_*j_{\tilde X})\circ\mu_0
=\pi^*\circ j_X
=(R\pi_*j_{\tilde X})\circ\varphi_0^{\mathbb C}.
\]
Since $R\pi_*j_{\tilde X}$ is an isomorphism, $\mu_0=\varphi_0^{\mathbb C}$.

\smallskip
\phantomsection\label{step-generalized-bc-3}
\noindent\emph{Step 3: the projective-bundle classes.}
Recall the morphisms $\alpha_k$ from \eqref{derived-eq-6-1}, and set
\[
\lambda_r=-(\iota_*\alpha_{r-1})[-2]:
\iota_*\mathbb C_Z[-2r]\longrightarrow \iota_*R\varpi_*\mathbb C_E[-2].
\]
Let $\rho_E=\iota_*R\varpi_*j_E$, using the canonical identification
$R\varpi_*\mathcal E_E^\bullet\simeq \varpi_*\mathcal E_E^\bullet$.
For $k=r-1$, adjunction and de Rham comparison give
\[
\begin{adjustbox}{max width=\linewidth}
\small$
\begin{tikzcd}[ampersand replacement=\&, column sep=1.8em, row sep=1.8em]
\operatorname{Hom}_{D^b(\mathbb C_Z)}
 (\mathbb C_Z[-2k],R\varpi_*\mathbb C_E)
 \arrow[r,"\sim"] \arrow[d,"{(R\varpi_*j_E)\circ-}"']
\& \operatorname{Hom}_{D^b(\mathbb C_E)}
 (\mathbb C_E[-2k],\mathbb C_E)
 \arrow[r,"\sim"] \arrow[d,"{j_E\circ-}"']
\& H^{2k}(E;\mathbb C) \arrow[d,"{\mathbb H^{2k}(j_E)}"] \\
\operatorname{Hom}_{D^b(\mathbb C_Z)}
 (\mathbb C_Z[-2k],R\varpi_*\mathcal E_E^\bullet)
 \arrow[r,"\sim"']
\& \operatorname{Hom}_{D^b(\mathbb C_E)}
 (\mathbb C_E[-2k],\mathcal E_E^\bullet)
 \arrow[r,"\sim"']
\& \mathbb H^{2k}(E,\mathcal E_E^\bullet).
\end{tikzcd}
$
\end{adjustbox}
\]
By definition, $\alpha_k$ corresponds to $\theta_E^k$, whose de Rham image
is $[\theta_E^k]$. The composite $u_k\circ j_Z[-2k]$ is adjoint to
\[
\mathbb C_E[-2k]\longrightarrow\mathcal E_E^\bullet,
\qquad 1\longmapsto\theta_E^k,
\]
and corresponds to the same class. Hence
\[
(R\varpi_*j_E)\circ\alpha_k=u_k\circ j_Z[-2k]
\quad\text{in }D^b(\mathbb C_Z).
\]
Applying $\iota_*$ and shifting by $[-2]$ yields
\[
\rho_E[-2]\circ\lambda_r
=-(\iota_*u_{r-1})[-2]\circ(\iota_*j_Z)[-2r].
\]
Define
\[
\mu_{\theta_E}:\iota_*\varpi_*\mathcal E_E^\bullet[-2]
\xrightarrow{\,\wedge\theta_E\,}\iota_*\varpi_*\mathcal E_E^\bullet
\xrightarrow{\,\iota_*q_Z\,}\iota_*\mathcal Q_Z^\bullet.
\]
Since $\mu_{\theta_E}\circ(\iota_*u_{r-1})[-2]=v_r$, we obtain
\[
(-\mu_{\theta_E})\circ\rho_E[-2]\circ\lambda_r
=v_r\circ(\iota_*j_Z)[-2r].
\]

\smallskip
\phantomsection\label{step-generalized-bc-4}
\noindent\emph{Step 4: the Gysin comparison and the exceptional components.}
Let
\[
G:\iota_*R\varpi_*\mathbb C_E[-2]
\xrightarrow{\sim}R\pi_*\jmath_*\mathbb C_E[-2]
\xrightarrow{R\pi_*\mathrm{Gy}_E}R\pi_*\mathbb C_{\tilde{X}}
\xrightarrow{\,R\pi_*j_{\tilde X}\,}\pi_*\mathcal E_{\tilde{X}}^\bullet.
\]
By the definitions of $\gamma$ and $\varphi_r^{\mathbb C}$,
\[
G\circ\lambda_r
=(R\pi_*j_{\tilde X})\circ\gamma\circ\lambda_r
=(R\pi_*j_{\tilde X})\circ\varphi_r^{\mathbb C}.
\]
We claim that
\[
h\circ q_X\circ G=(-\mu_{\theta_E})\circ\rho_E[-2].
\]
Put $N=N_{E/\tilde{X}}$ and let
$\mathrm{res}:\mathbb C_{\tilde{X}}\to \jmath_*\mathbb C_E$ be restriction.
The identification
\[
\operatorname{Hom}_{D^b(\mathbb C_{\tilde{X}})}
 (\jmath_*\mathbb C_E[-2],\jmath_*\mathbb C_E)
\cong H^2(E;\mathbb C)
\]
sends $\mathrm{res}\circ\mathrm{Gy}_E$ to
$\jmath^*\jmath_{\mathrm{Gys}}(1)$. Consequently, by
\eqref{eq-gysin-self-intersection},
\[
\mathrm{res}\circ\mathrm{Gy}_E=\jmath_*\bigl(c_1(N)\smile-\bigr).
\]
Here $c_1(N)\smile-:\mathbb C_E[-2]\to\mathbb C_E$ denotes the
morphism corresponding to $c_1(N)$ under the canonical identification
$\operatorname{Hom}_{D^b(\mathbb C_E)}(\mathbb C_E[-2],\mathbb C_E)
\cong H^2(E;\mathbb C)$; on cohomology it sends a class $\beta$ to
$c_1(N)\smile\beta$.
Since $N\simeq\mathcal O_E(-1)$, its first Chern class is represented
by $-\theta_E$. Let
$r_E:\pi_*\mathcal E_{\tilde{X}}^\bullet\to \iota_*\varpi_*\mathcal E_E^\bullet$
be restriction of forms. If
$\mathrm{res}_{\mathcal E}:\mathcal E_{\tilde{X}}^\bullet
\to\jmath_*\mathcal E_E^\bullet$ denotes restriction on $\tilde{X}$, then
\[
\mathrm{res}_{\mathcal E}\circ j_{\tilde{X}}
=(\jmath_*j_E)\circ\mathrm{res},
\qquad
j_E\circ(c_1(N)\smile-)
=-(\wedge\theta_E)\circ j_E[-2]
\]
in $D^b(\mathbb C_{\tilde{X}})$ and $D^b(\mathbb C_E)$, respectively.
Using $R\pi_*\jmath_*=\iota_*R\varpi_*$ and
$\rho_E=\iota_*R\varpi_*j_E$, we obtain in $D^b(\mathbb C_X)$
\[
\begin{aligned}
r_E\circ G
&=R\pi_*\bigl(\mathrm{res}_{\mathcal E}\circ j_{\tilde{X}}
  \circ\mathrm{Gy}_E\bigr)\\
&=R\pi_*\bigl((\jmath_*j_E)\circ\mathrm{res}
  \circ\mathrm{Gy}_E\bigr)\\
&=\iota_*R\varpi_*\bigl(j_E\circ(c_1(N)\smile-)\bigr)\\
&=-\iota_*R\varpi_*\bigl((\wedge\theta_E)\circ j_E[-2]\bigr)\\
&=-\iota_*\varpi_*(\wedge\theta_E)\circ\rho_E[-2].
\end{aligned}
\]
Since $\jmath^*\circ\pi^*=\varpi^*\circ\iota^*$, we have
$h\circ q_X=\iota_*q_Z\circ r_E$.
Composing the last equality with $\iota_*q_Z$ proves the claim, or
equivalently the commutativity of
\[
\begin{adjustbox}{max width=\linewidth}
\small$
\begin{tikzcd}[ampersand replacement=\&, column sep=1.8em, row sep=1.8em]
\iota_*R\varpi_*\mathbb C_E[-2]
 \arrow[r,"G"] \arrow[d,"{\rho_E[-2]}"']
\& \pi_*\mathcal E_{\tilde{X}}^\bullet \arrow[d,"h\circ q_X"] \\
\iota_*\varpi_*\mathcal E_E^\bullet[-2]
 \arrow[r,"-\mu_{\theta_E}"']
\& \iota_*\mathcal Q_Z^\bullet.
\end{tikzcd}
$
\end{adjustbox}
\]
Together with \hyperref[step-generalized-bc-3]{Step~3}, this gives
\begin{equation}\label{eq-exceptional-quotient-component}
h\circ q_X\circ G\circ\lambda_r=v_r\circ(\iota_*j_Z)[-2r].
\end{equation}
Since $\iota$ has complex codimension $c$, the complex orientation gives
$\iota^!\mathbb C_X\cong\mathbb C_Z[-2c]$. Thus, for $1\le r\le c-1$,
\[
\begin{aligned}
\operatorname{Hom}_{D^b(\mathbb C_X)}
 (\iota_*\mathbb C_Z[-2r],\mathbb C_X)
&\cong\operatorname{Hom}_{D^b(\mathbb C_Z)}
 (\mathbb C_Z[-2r],\mathbb C_Z[-2c])\\
&\cong H^{2r-2c}(Z;\mathbb C)=0.
\end{aligned}
\]
As $\Phi_X\circ j_X:\mathbb C_X\to\mathcal D\mathcal E_X^\bullet$
is a quasi-isomorphism, $\pi_*\circ G\circ\lambda_r=0$.
Now by \eqref{eq-exceptional-quotient-component}, we obtain
\[
\begin{aligned}
\Psi\circ(R\pi_*j_{\tilde X})\circ\varphi_r^{\mathbb C}
&=\Psi\circ G\circ\lambda_r\\
&=\bigl(0,v_r\circ(\iota_*j_Z)[-2r]\bigr)\\
&=\Psi\circ(R\pi_*j_{\tilde X})\circ\mu_r.
\end{aligned}
\]
Since both $\Psi$ and $R\pi_*j_{\tilde X}$ are isomorphisms, we
conclude that $\varphi_r^{\mathbb C}=\mu_r$ for $1\le r\le c-1$.
Together with \hyperref[step-generalized-bc-2]{Step~2}, this proves $\varphi^{\mathbb C}=\mu$.

\smallskip
\phantomsection\label{step-generalized-bc-5}
\noindent\emph{Step 5: the two coefficient comparison squares.}
By Lemma~\ref{lem-constant-sheaf-coefficient-change},
\[
\begin{adjustbox}{max width=\linewidth}
\small$
\begin{tikzcd}[ampersand replacement=\&, column sep=1.8em, row sep=1.8em]
S_X\oplus\displaystyle\bigoplus_{r=1}^{c-1}\iota_*S_Z[-2r]
 \arrow[r,"\varphi^S"] \arrow[d,"f_*"']
\& R\pi_*S_{\tilde{X}} \arrow[d,"R\pi_*f_*"] \\
\mathbb C_X\oplus\displaystyle\bigoplus_{r=1}^{c-1}\iota_*\mathbb C_Z[-2r]
 \arrow[r,"\varphi^{\mathbb C}"']
\& R\pi_*\mathbb C_{\tilde{X}}
\end{tikzcd}
$
\end{adjustbox}
\]
commutes, and the analogous statement holds for $g$.
Using $\varphi^{\mathbb C}=\mu$ from \hyperref[step-generalized-bc-4]{Step~4} and the truncation
compatibility of Corollary~\ref{cor-smooth-truncated-blowup}, we obtain
\[
\begin{adjustbox}{max width=\linewidth}
\small$
\begin{tikzcd}[ampersand replacement=\&, column sep=1.8em, row sep=1.8em]
\Omega_X^{[0,p-1]}\oplus
 \displaystyle\bigoplus_{r=1}^{c-1}\iota_*\Omega_Z^{[0,p-1-r]}[-2r]
 \arrow[r,"\eta"]
\& R\pi_*\Omega_{\tilde{X}}^{[0,p-1]} \\
S_X\oplus\displaystyle\bigoplus_{r=1}^{c-1}\iota_*S_Z[-2r]
 \arrow[r,"\varphi^S"] \arrow[u,"f_*"']
\& R\pi_*S_{\tilde{X}} \arrow[u,"f_*"].
\end{tikzcd}
$
\end{adjustbox}
\]
Here $f_*$ also denotes the maps induced by $f$ into the truncated
holomorphic complexes, including the shifted direct sums on the left.
Repeating the argument for antiholomorphic forms and $g$ gives
\[
\begin{adjustbox}{max width=\linewidth}
\small$
\begin{tikzcd}[ampersand replacement=\&, column sep=1.8em, row sep=1.8em]
\overline{\Omega_X^{[0,q-1]}}\oplus
 \displaystyle\bigoplus_{r=1}^{c-1}\iota_*\overline{\Omega_Z^{[0,q-1-r]}}[-2r]
 \arrow[r,"\bar\eta"]
\& R\pi_*\overline{\Omega_{\tilde{X}}^{[0,q-1]}} \\
S_X\oplus\displaystyle\bigoplus_{r=1}^{c-1}\iota_*S_Z[-2r]
 \arrow[r,"\varphi^S"] \arrow[u,"g_*"']
\& R\pi_*S_{\tilde{X}} \arrow[u,"g_*"]
\end{tikzcd}
$
\end{adjustbox}
\]
with the analogous convention for $g_*$.

\smallskip
\phantomsection\label{step-generalized-bc-6}
\noindent\emph{Step 6: passage to the Bott--Chern cones.}
Set
\[
\mathcal C_X=\Omega_X^{[0,p-1]}\oplus\overline{\Omega_X^{[0,q-1]}},
\qquad
\mathcal D_Z^r=\Omega_Z^{[0,p-1-r]}\oplus\overline{\Omega_Z^{[0,q-1-r]}}.
\]
The two squares in \hyperref[step-generalized-bc-5]{Step~5} combine into a commutative square with
horizontal maps induced by $\Delta_{f,g}$ into the corresponding truncated complexes.
The axiom for morphisms of distinguished triangles supplies a morphism
\[
b:\mathcal B_X^{p,q}(f,g)\oplus
\bigoplus_{r=1}^{c-1}\iota_*\mathcal B_Z^{p-r,q-r}(f,g)[-2r]
\longrightarrow R\pi_*\mathcal B_{\tilde{X}}^{p,q}(f,g)
\]
completing the following commutative diagram:
\[
\begin{adjustbox}{max width=\linewidth}
\small$
\begin{tikzcd}[ampersand replacement=\&, column sep=small, row sep=small]
S_X\oplus \bigoplus_{r=1}^{c-1}\iota_*S_Z[-2r]
 \arrow[r] \arrow[d,"\varphi^S"']
\& \mathcal C_X\oplus \bigoplus_{r=1}^{c-1}\iota_*\mathcal D_Z^r[-2r]
 \arrow[r] \arrow[d,"{(\eta,\bar\eta)}"']
\& \mathcal B_X^{p,q}(f,g)[1]\oplus \bigoplus_{r=1}^{c-1}\iota_*\mathcal B_Z^{p-r,q-r}(f,g)[-2r][1]
 \arrow[r] \arrow[d,"{b[1]}"']
\& S_X[1]\oplus \bigoplus_{r=1}^{c-1}\iota_*S_Z[-2r][1]
 \arrow[d,"{\varphi^S[1]}"'] \\
R\pi_*S_{\tilde{X}} \arrow[r]
\& R\pi_*\Omega_{\tilde{X}}^{[0,p-1]}\oplus  R\pi_*\overline{\Omega_{\tilde{X}}^{[0,q-1]}}
 \arrow[r]
\& R\pi_*\mathcal B_{\tilde{X}}^{p,q}(f,g)[1] \arrow[r]
\& R\pi_*S_{\tilde{X}}[1]
\end{tikzcd}
$
\end{adjustbox}
\]
The rows are distinguished triangles. The first two vertical maps are
isomorphisms, so $b$ is an isomorphism by \cite[Corollary 1.5.5]{KS}.
This proves the blow-up decomposition.
\end{proof}

Applying hypercohomology $\mathbb H^\bullet(X,-)$ to both sides of
\eqref{eq-derived-generalized-bc-blowup}, we obtain the following corollary.

\begin{Cor}\label{cor-generalized-bc-blowup}
In the setting of Theorem~\ref{thm-generalized-bc-blowup}, define the
\emph{generalized Bott--Chern cohomology groups} by
\[
H^{p,q}_{\mathrm{BC}}(X;f,g)
:=\mathbb H^{p+q}(X,\mathcal B_X^{p,q}(f,g)).
\]
Then
\[
H^{p,q}_{\mathrm{BC}}(\tilde{X};f,g)
\cong H^{p,q}_{\mathrm{BC}}(X;f,g)
\oplus\bigoplus_{r=1}^{c-1}H^{p-r,q-r}_{\mathrm{BC}}(Z;f,g).
\]
\end{Cor}

\begin{Rem}
For compact $X$, taking $S=\mathbb C$ and $f=g=\mathrm{id}_{\mathbb C}$
in Theorem~\ref{thm-generalized-bc-blowup} and passing to hypercohomology
recovers the blow-up formula of S.~Yang--X.~Yang \cite[Theorem 3.7]{SY}.
For ordinary Bott--Chern cohomology with complex coefficients, Meng
\cite[Proposition 4.15]{Mn20BC} established a blow-up formula without
assuming compactness.
For compact $X$, taking $S=\mathbb Z$ and
$f,g:\mathbb Z\hookrightarrow\mathbb C$ to be the natural inclusions
recovers the integral Bott--Chern blow-up
formula of Y.~Chen--S.~Yang \cite[Theorem 1.2]{CSY} and X.~Wu
\cite[Proposition 12]{Wu}.
\end{Rem}

\section{\texorpdfstring{The $\partial\bar\partial$-property relative to $(f,g)$}{The d-dbar-property relative to (f,g)}}
\label{sec-generalized-ddbar}

Throughout this section, $X$ is a compact complex manifold,
$S$ is a commutative ring with unit, and $f,g:S\longrightarrow\mathbb C$
are ring homomorphisms. We introduce a notion of the $\partial\bar\partial$-property
relative to the ordered pair $(f,g)$, characterize it by the Fr\"olicher
spectral sequence and the Hodge filtrations, and establish a blow-up
formula for the kernel of the comparison map.

\subsection{Definition and the filtration criterion}

We define the relative $\partial\bar\partial$-property through the
injectivity of the Bott--Chern comparison maps. We then show that both
coefficient maps must be surjective and characterize the property in
terms of $E_1$-degeneration and the Hodge filtrations, distinguishing
whether $f$ and $g$ have the same kernel.

\begin{De}\label{def-relative-ddbar}
Let $\delta_{p,q}:\mathcal B_X^{p,q}(f,g)\longrightarrow S_X$ be the
natural projection from the shifted mapping cone to its coefficient
complex. We say that $X$ satisfies the \emph{$\partial\bar\partial$-property
relative to $(f,g)$} if the \emph{comparison maps}
\[
\delta_{p,q}^*:H^{p,q}_{\mathrm{BC}}(X;f,g)
\longrightarrow H^{p+q}(X;S)
\]
are injective for all $p,q\in\mathbb Z$.
\end{De}

The cone triangle gives an exact sequence
\[
\begin{aligned}
H^{p+q-1}(X;S)&\xrightarrow{\pi_{p,q}^{f,g}}
\mathbb H^{p+q-1}(X,\Omega_X^{[0,p-1]})
\oplus\mathbb H^{p+q-1}(X,\overline{\Omega_X^{[0,q-1]}})\\
&\longrightarrow H^{p,q}_{\mathrm{BC}}(X;f,g)
\xrightarrow{\delta_{p,q}^*}H^{p+q}(X;S).
\end{aligned}
\]
Thus the relative property is equivalent to the surjectivity of
$\pi_{p,q}^{f,g}$ for every $p,q$. Write $\iota_*$ and $\jmath_*$ for the maps
from complex de Rham cohomology to the holomorphic and antiholomorphic
truncated hypercohomology groups, respectively, induced by the inclusions
of constant functions.
For each degree $k$, let
\[
f_*,g_*:H^k(X;S)\longrightarrow H^k(X;\mathbb C)
\]
be the maps induced by the morphisms of constant
sheaves $S_X\to\mathbb C_X$ associated with $f$ and $g$, respectively.
Then
\[
\pi_{p,q}^{f,g}=(\iota_*f_*,\jmath_*g_*).
\]

\begin{Pro}\label{prop-relative-ddbar-coefficients}
If $X$ satisfies the $\partial\bar\partial$-property relative to $(f,g)$,
then $f$ and $g$ are surjective.
\end{Pro}
\begin{proof}
Take $(p,q)=(1,0)$ and $(0,1)$ in the preceding exact sequence.
Since $X$ is compact and connected, holomorphic and antiholomorphic
functions on $X$ are constant. The corresponding maps $\pi_{p,q}^{f,g}$ in
degree zero identify with $f,g:S\to\mathbb C$, respectively, and hence
$f$ and $g$ are surjective.
\end{proof}

\begin{Le}\label{lem-relative-quotient-surjectivity}
Let $V$ be a vector space, $F,F'\subset V$ subspaces, and
$\sigma:V\longrightarrow V$ an automorphism of the underlying
abelian group. The map
\[
q:V\longrightarrow V/F\oplus V/F',\qquad
q(y)=(y\bmod F,\sigma(y)\bmod F')
\]
is surjective if and only if $F+\sigma^{-1}(F')=V$.
\end{Le}
\begin{proof}
Given $v_1,v_2\in V$, the pair $(v_1\bmod F,v_2\bmod F')$ has a preimage if
and only if
\[
(v_1+F)\cap\sigma^{-1}(v_2+F')\ne\varnothing.
\]
This is equivalent to
$v_1-\sigma^{-1}(v_2)\in F+\sigma^{-1}(F')$.
As $v_1,v_2$ vary, their differences $v_1-\sigma^{-1}(v_2)$ range over $V$,
which proves the assertion.
\end{proof}

\begin{Th}\label{thm-relative-ddbar-criterion}
Assume that $f,g:S\longrightarrow\mathbb C$ are surjective. Put
$V^k=H^k(X;\mathbb C)$, and let $F^\bullet V^k$ be its Hodge filtration.
\begin{enumerate}[(1)]
\item\label{thm-relative-ddbar-distinct-kernels} If $\ker f\ne\ker g$, then $X$ satisfies the
$\partial\bar\partial$-property relative to $(f,g)$ if and only if its
Fr\"olicher spectral sequence degenerates at $E_1$.
\item\label{thm-relative-ddbar-equal-kernels} If $\ker f=\ker g$, let $\sigma\in\operatorname{Aut}(\mathbb C)$
be the unique field automorphism with $g=\sigma\circ f$, and
$\sigma_k:V^k\longrightarrow V^k$ induced by $\sigma$. Then $X$ satisfies the $\partial\bar\partial$-property
relative to $(f,g)$ if and only if its Fr\"olicher spectral sequence
degenerates at $E_1$ and
\[
F^pV^{p+q-1}
+\sigma_{p+q-1}^{-1}\!\left(\overline{F^qV^{p+q-1}}\right)
=V^{p+q-1}
\]
for every $p,q\in\mathbb Z$.
\end{enumerate}
Here the bar denotes conjugation on de Rham cohomology with respect to
its real structure. The automorphisms $\sigma_k$ need not be
$\mathbb C$-linear.
\end{Th}
\begin{proof}
\phantomsection\label{step-relative-ddbar-1}
\emph{Step 1: surjectivity under change of coefficients.}
Suppose first that $\ker f\ne\ker g$. These kernels are distinct
maximal ideals, so $\ker f+\ker g=S$. By the Chinese remainder theorem,
\[
S\longrightarrow S/\ker f\times S/\ker g
\xrightarrow{\sim}\mathbb C\times\mathbb C
\]
is surjective. Thus $h=(f,g)$ is surjective.

We claim that $h_*=(f_*,g_*)$ is surjective on cohomology in every degree.
Put $H_k=H_k(X;\mathbb Z)$. The universal coefficient theorem
\cite[Theorem 3.2]{Ha02} gives a commutative diagram with exact rows:
\[
\begin{adjustbox}{max width=\linewidth}
\small$
\begin{tikzcd}[ampersand replacement=\&, column sep=small, row sep=2em]
0 \arrow[r]
\& \operatorname{Ext}^1_{\mathbb Z}(H_{k-1},S)
 \arrow[r] \arrow[d,"{\operatorname{Ext}^1(h)}"']
\& H^k(X;S) \arrow[r] \arrow[d,"h_*"]
\& \operatorname{Hom}_{\mathbb Z}(H_k,S)
 \arrow[r] \arrow[d,"{\operatorname{Hom}(h)}"]
\& 0 \\
0 \arrow[r]
\& \operatorname{Ext}^1_{\mathbb Z}(H_{k-1},\mathbb C\times\mathbb C)
 \arrow[r]
\& H^k(X;\mathbb C)\oplus H^k(X;\mathbb C) \arrow[r]
\& \operatorname{Hom}_{\mathbb Z}(H_k,\mathbb C\times\mathbb C)
 \arrow[r]
\& 0.
\end{tikzcd}
$
\end{adjustbox}
\]
Since $X$ is compact, $H_k$ is finitely generated. Write
$H_k\cong\mathbb Z^m\oplus T$ with $T$ finite. Every homomorphism
$H_k\to\mathbb C\times\mathbb C$ vanishes on $T$ and lifts through $h$
by choosing lifts of the images of a basis of $\mathbb Z^m$.
Hence $\operatorname{Hom}(h)$ is surjective. The group
$\mathbb C\times\mathbb C$ is divisible and thus injective over
$\mathbb Z$, so the lower left term is zero. A diagram chase proves
that $h_*$ is surjective. The same argument with target $\mathbb C$
shows that any surjective coefficient map $f:S\to\mathbb C$ induces
surjections $f_*:H^k(X;S)\to V^k$.

\smallskip
\phantomsection\label{step-relative-ddbar-2}
\emph{Step 2: the truncation maps and $E_1$-degeneration.}
We claim that the maps $\iota_*$ are surjective for every truncation and
every degree if and only if the Fr\"olicher spectral sequence degenerates
at $E_1$. If it degenerates, the maps to the truncated complexes give
canonical identifications
\[
\mathbb H^k(X,\Omega_X^{[0,p-1]})\cong V^k/F^pV^k,
\qquad
\mathbb H^k(X,\overline{\Omega_X^{[0,q-1]}})
\cong V^k/\overline{F^qV^k}.
\]
Under these identifications, $\iota_*$ and $\jmath_*$ are quotient maps and
therefore surjective.

Conversely, the exact sequence
\[
0\longrightarrow\Omega_X^p[-p]
\longrightarrow\Omega_X^{[0,p]}
\longrightarrow\Omega_X^{[0,p-1]}\longrightarrow0
\]
gives the long exact sequence
\[
\begin{aligned}
\cdots\longrightarrow H^{k-p}(X,\Omega_X^p)
&\longrightarrow\mathbb H^k(X,\Omega_X^{[0,p]})
\xrightarrow{\rho^k}\mathbb H^k(X,\Omega_X^{[0,p-1]})\\
&\longrightarrow H^{k-p+1}(X,\Omega_X^p)\longrightarrow\cdots.
\end{aligned}
\]
The map from $V^k$ to the last truncated hypercohomology group factors
through $\rho^k$, so the assumed surjectivity of $\iota_*$ implies the
surjectivity of every $\rho^k$. Applying this also in degree $k-1$
gives short exact sequences
\[
0\longrightarrow H^{k-p}(X,\Omega_X^p)
\longrightarrow\mathbb H^k(X,\Omega_X^{[0,p]})
\longrightarrow\mathbb H^k(X,\Omega_X^{[0,p-1]})
\longrightarrow0.
\]
Induction on $p$, ending at $p=\dim_{\mathbb C}X$, yields
\[
\dim_{\mathbb C}V^k
=\sum_{p=0}^{\dim_{\mathbb C}X}h^{p,k-p}(X)
\]
for every $k$. This is equivalent to $E_1$-degeneration.

If the relative property holds, the surjectivity of $\pi_{p,q}^{f,g}$
implies that of each component, hence that of $\iota_*$ and $\jmath_*$.
Thus $E_1$-degeneration is necessary in both cases.
When the kernels are distinct,
\hyperref[step-relative-ddbar-1]{Step~1} shows that
\[
\pi_{p,q}^{f,g}=(\iota_*\oplus \jmath_*)\circ(f_*,g_*)
\]
is surjective exactly when both $\iota_*$ and $\jmath_*$ are surjective.
This proves (\ref{thm-relative-ddbar-distinct-kernels}).

\smallskip
\phantomsection\label{step-relative-ddbar-3}
\emph{Step 3: coincident kernels.}
If $\ker f=\ker g$, define $\sigma(f(s))=g(s)$. This is well-defined
and is the unique field automorphism satisfying $\sigma\circ f=g$.
This gives
$g_*=\sigma_k\circ f_*$. For $k=p+q-1$ we therefore have
\[
\begin{aligned}
H^k(X;S)\xrightarrow{f_*}V^k
\xrightarrow{(\iota_*,\jmath_*\sigma_k)}
\mathbb H^k(X,\Omega_X^{[0,p-1]})
\oplus\mathbb H^k(X,\overline{\Omega_X^{[0,q-1]}}).
\end{aligned}
\]
Since $f_*$ is surjective by
\hyperref[step-relative-ddbar-1]{Step~1}, the composite is surjective if
and only if the second map is surjective. Under $E_1$-degeneration,
\hyperref[step-relative-ddbar-2]{Step~2} identifies the latter with
\[
V^k\longrightarrow V^k/F^pV^k\oplus V^k/\overline{F^qV^k},
\qquad x\longmapsto
(x\bmod F^pV^k,\sigma_k(x)\bmod\overline{F^qV^k}).
\]
Lemma~\ref{lem-relative-quotient-surjectivity} now gives
(\ref{thm-relative-ddbar-equal-kernels}).
\end{proof}

\subsection{Examples with relative \texorpdfstring{$\partial\bar\partial$}{partial-bar-partial}-property}

The following examples illustrate how the relative property depends on
the coefficient maps. Equal surjective maps recover the classical
$\partial\bar\partial$-lemma, whereas pairing the identity with complex
conjugation imposes the additional vanishing of off-diagonal Hodge
numbers on a $\partial\bar\partial$-manifold.

\begin{Ex}\label{ex-relative-equal-maps}
Let $f:S\longrightarrow\mathbb C$ be surjective. Then $X$ satisfies
the $\partial\bar\partial$-property relative to $(f,f)$ if and only if
$X$ satisfies the $\partial\bar\partial$-lemma. Indeed, the surjectivity of $f_*$ proved above shows that
$\pi_{p,q}^{f,f}=(\iota_*,\jmath_*)\circ f_*$ is surjective exactly when
$\pi_{p,q}^{\mathrm{id},\mathrm{id}}=(\iota_*,\jmath_*)$ is surjective.
For $S=\mathbb C$ and $f=g=\mathrm{id}$, multiplying the antiholomorphic
summand by $-1$ identifies the cone with the usual complex-coefficient
Bott--Chern complex. Its comparison map is the natural map
$H^{p,q}_{\mathrm{BC}}(X)\to H^{p+q}_{dR}(X;\mathbb C)$.
Injectivity in every bidegree is the
$\partial\bar\partial$-lemma; see \cite[\S~1]{Sch} and \cite[\S~2]{S1}.
\end{Ex}

\begin{Ex}\label{ex-relative-conjugation}
Suppose that $X$ satisfies the $\partial\bar\partial$-lemma.
Let $\overline{\mathrm{id}}:\mathbb C\to\mathbb C$ denote complex
conjugation. Then $X$ satisfies the $\partial\bar\partial$-property
relative to $(\mathrm{id},\overline{\mathrm{id}})$ if and only if
$h^{p,q}(X)=0$ whenever $p\ne q$.
To see this, note that $\sigma_k$ is conjugation on $V^k$. Hence
$\sigma_k^{-1}(\overline{F^sV^k})=F^sV^k$, and the filtration condition
in Theorem~\ref{thm-relative-ddbar-criterion}.(\ref{thm-relative-ddbar-equal-kernels}) becomes
\[
F^rV^{r+s-1}+F^sV^{r+s-1}=V^{r+s-1}.
\]
Suppose that an off-diagonal Hodge number is nonzero. By Hodge symmetry
we may choose $p>q$ with $h^{q,p}(X)\ne0$. Set $r=p$ and $s=q+1$.
Since $F^pV^{p+q}\subseteq F^{q+1}V^{p+q}$, the displayed condition
would imply $F^{q+1}V^{p+q}=V^{p+q}$. But the Hodge decomposition has
a nonzero summand $H^{q,p}(X)$ outside $F^{q+1}$, a contradiction.

Conversely, if all off-diagonal Hodge numbers vanish, then $V^k=0$
for odd $k$, while for $k=2p$ we have $V^{2p}=H^{p,p}(X)$ and
\[
F^rV^{2p}=\begin{cases}V^{2p},&r\le p,\\0,&r>p.\end{cases}
\]
If $r+s=2p+1$, at least one of $r,s$ is at most $p$, so the required
sum equals $V^{2p}$. In particular, an elliptic curve does not satisfy
this relative property, since $h^{1,0}=h^{0,1}=1\ne0$.
\end{Ex}

\subsection{The blow-up formula for the comparison kernel}

We apply the compatible derived blow-up decompositions to the kernels
of the comparison maps, which measure the failure of the relative
$\partial\bar\partial$-property. The resulting formula gives a blow-up
criterion and implies bimeromorphic invariance for compact complex
surfaces. A projective threefold example shows that this invariance
need not persist in dimension three.

\begin{Pro}\label{prop-generalized-ddbar-blowup}
Let $\pi:\tilde{X}\longrightarrow X$ be the blow-up along a closed
complex submanifold $Z\subset X$ of codimension $c\ge2$. Define the
\emph{comparison kernel} by
\[
\mathcal K_X^{p,q}(f,g):=\ker\delta_{p,q}^*.
\]
Then there is an isomorphism
\[
\mathcal K_{\tilde{X}}^{p,q}(f,g)
\cong\mathcal K_X^{p,q}(f,g)\oplus
\bigoplus_{r=1}^{c-1}\mathcal K_Z^{p-r,q-r}(f,g).
\]
In particular, $\tilde{X}$ satisfies the $\partial\bar\partial$-property
relative to $(f,g)$ if and only if both $X$ and $Z$ satisfy it.
\end{Pro}
\begin{proof}
By Theorem~\ref{thm-generalized-bc-blowup}, we have the commutative diagram
\[
\begin{adjustbox}{max width=\linewidth}
\small$
\begin{tikzcd}[ampersand replacement=\&, column sep=small, row sep=1.8em]
H^{p,q}_{\mathrm{BC}}(X;f,g)\oplus
 \bigoplus_{r=1}^{c-1}H^{p-r,q-r}_{\mathrm{BC}}(Z;f,g)
 \arrow[r] \arrow[d,"b"']
\& H^{p+q}(X;S)\oplus
 \bigoplus_{r=1}^{c-1}H^{p+q-2r}(Z;S)
 \arrow[d,"\varphi^S"] \\
H^{p,q}_{\mathrm{BC}}(\tilde{X};f,g) \arrow[r]
\& H^{p+q}(\tilde{X};S).
\end{tikzcd}
$
\end{adjustbox}
\]
The horizontal maps are the corresponding comparison maps
$\delta_{p,q}^*$ and their shifted direct sums. Since $b$ and
$\varphi^S$ are isomorphisms, $b$ identifies the kernel of the top
horizontal map with that of the bottom horizontal map. This gives the
asserted decomposition. Letting $p,q$ vary and using the summand $r=1$
for $Z$ proves the final assertion.
\end{proof}

\begin{Cor}\label{cor-relative-ddbar-surfaces}
For any commutative ring $S$ and ring maps $f,g:S\to\mathbb C$, the
$\partial\bar\partial$-property relative to $(f,g)$ is a bimeromorphic
invariant of smooth compact complex surfaces.
\end{Cor}
\begin{proof}
If either $f$ or $g$ is not surjective, no such surface satisfies the
relative property by Proposition~\ref{prop-relative-ddbar-coefficients}.
Suppose therefore that both maps are surjective. If $\ker f\ne\ker g$,
a point satisfies the relative property by
Theorem~\ref{thm-relative-ddbar-criterion}.(\ref{thm-relative-ddbar-distinct-kernels}).
If $\ker f=\ker g$, it satisfies the relative property by
Theorem~\ref{thm-relative-ddbar-criterion}.(\ref{thm-relative-ddbar-equal-kernels}):
its spectral sequence degenerates at $E_1$, and its filtration is
$F^0V^0=V^0$, $F^1V^0=0$, so the displayed condition is automatic.
Proposition~\ref{prop-generalized-ddbar-blowup}
then shows that blowing up a point preserves the relative property in
both directions. By the weak factorization theorem, any bimeromorphic map
between smooth compact complex surfaces factors into a finite sequence
of blow-ups and blow-downs at points.
The assertion follows by applying the point blow-up criterion at each step.
\end{proof}

\begin{Ex}\label{ex-relative-ddbar-threefold-blowup}
The $\partial\bar\partial$-property relative to
$(\mathrm{id},\overline{\mathrm{id}}):\mathbb C\to\mathbb C$ is not a
bimeromorphic invariant of smooth compact complex threefolds, even within
the projective category. Indeed, let $Z\subset\mathbb P^2\subset\mathbb P^3$
be a smooth plane cubic curve, and
$\pi:\tilde{X}=\Bl_Z\mathbb P^3\to\mathbb P^3$ its blow-up.
Since $\mathbb P^3$ is K\"ahler and has only diagonal Hodge numbers,
Example~\ref{ex-relative-conjugation} shows that it satisfies the relative
property, whereas the elliptic curve $Z$ does not. Since $Z$ has codimension two,
Proposition~\ref{prop-generalized-ddbar-blowup} implies that $\tilde{X}$
does not satisfy the relative property.
\end{Ex}

\section{Limits of blow-up and K\"unneth formulae}
\label{counterexample-part}

This section tests the hypotheses and possible extensions of the preceding
blow-up decompositions.  Three counterexamples show, respectively, that
locally free Dolbeault coefficients cannot generally be replaced by
coherent ones, that the ordinary de Rham formula does not extend unchanged
to singular spaces, and that Bott--Chern cohomology admits no naive
tensor-product K\"unneth formula.

\subsection{Failure beyond locally free coefficients}
\label{subsec-failure-locally-free-coefficients}

We begin the limitations section by testing the coefficient hypothesis in
the Dolbeault blow-up formula.  An explicit ideal sheaf on $\mathbb P^2$
shows that local freeness cannot be replaced by coherence, even for a
torsion-free sheaf; a second argument explains why analogous extensions to
singular varieties also fail.

If $\pi:\tilde{X}\longrightarrow X$ is the blow-up of a compact complex
manifold $X$ of dimension $n$ along a complex submanifold $Z$, and if $W$ is
a locally free $\mathcal O_X$-module of finite rank, then for any
$0\leq p,q\leq n$ \cite[Theorem~1.2]{RYY20} gives a canonical
isomorphism 
\begin{equation}\label{counter-eq-1-1}
H^{p,q}(\tilde{X},\pi^*W)
\cong H^{p,q}(X,W)\oplus
\bigoplus_{r=1}^{c-1}H^{p-r,q-r}(Z,\iota^*W),
\end{equation}
where $\iota:Z\hookrightarrow X$ is the inclusion and
$c=\operatorname{codim}_X Z$.

Take $X=\mathbb P^2$, let $Z\subset X$ be a point, and $F=I_Z$
its ideal sheaf.  Write $\pi:\tilde{X}\longrightarrow X$ for the blow-up
at $Z$, and $\jmath:E\hookrightarrow\tilde{X}$ for the exceptional divisor,
where $E\cong\mathbb P^1$.  Although $F$ is coherent and torsion-free,
we will show that
\begin{equation}\label{counter-eq-1-2}
H^0(\tilde{X},\Omega^1_{\tilde{X}}\otimes\pi^*F)\cong\mathbb C,
\qquad H^0(X,\Omega_X^1\otimes F)=0.
\end{equation}
This contradicts the proposed extension of \eqref{counter-eq-1-1} for
$(p,q)=(1,0)$, since all exceptional summands then have negative second
degree and vanish.

We claim that there is an exact sequence
\begin{equation}\label{eq-ideal-pullback-torsion}
0\longrightarrow \jmath_*\mathcal O_E(-1)
\longrightarrow\pi^*F
\xrightarrow{\psi}\mathcal O_{\tilde{X}}(-E)
\longrightarrow0.
\end{equation}
Here \(\psi\) is induced by pulling back the inclusion \(F = I_Z \hookrightarrow \mathcal{O}_X\). Indeed,
 the image of $\psi$ is \(I_Z \mathcal{O}_{\tilde X} = I_E = \mathcal{O}_{\tilde X}(-E)\), so it is surjective. 
We identify the kernel of $\psi$ directly on the two standard blow-up charts.
Choose a coordinate neighborhood \(V\) of \(Z\) with coordinates \((x,y)\) centered at \(Z\). 
On the chart \(U \subset \pi^{-1}(V)\) with \(x = u\) and \(y = uv\), the presentation
\[
I_Z|_V = (x,y) \cong \mathcal{O}_V^{\oplus 2} / \mathcal{O}_V \cdot (-y,x)
\]
pulls back to
\[
\pi^* I_Z|_U \cong \mathcal{O}_U^{\oplus 2} / \mathcal{O}_U \cdot u(-v,1),
\]
and
\[
\psi_U([(a,b)]) = au + buv = u(a+bv).
\]
Thus \(\ker \psi_U\) is generated by \(e_U = [(-v,1)]\). On the other chart \(U'\subset \pi^{-1}(V)\) with 
\(x = st\) and \(y = t\), one similarly obtains
\[
\pi^* I_Z|_{U'} \cong \mathcal{O}_{U'}^{\oplus 2} / \mathcal{O}_{U'} \cdot t(-1,s), 
\qquad \psi_{U'}([(a,b)]) = t(as+b).
\]
Therefore \(\ker \psi_{U'}\) is generated by \(e_{U'} = [(-1,s)]\). 
On \(U \cap U'\), where \(s = v^{-1}\), the generators satisfy
\[
e_U = v e_{U'}.
\]
This is the transition relation for \(\mathcal{O}_E(-1)\) on \(E \cong \mathbb{P}^1\). 
Consequently \(\ker \psi \cong \jmath_* \mathcal{O}_E(-1)\),
proving \eqref{eq-ideal-pullback-torsion}.

Since the normal bundle of $E$ is $\mathcal O_E(-1)$, its conormal
sequence is
\[
0\longrightarrow\mathcal O_E(1)
\longrightarrow\Omega^1_{\tilde{X}}|_E
\longrightarrow\mathcal O_E(-2)\longrightarrow0.
\]
It splits because $H^1(E,\mathcal O_E(3))=0$.  Tensoring
\eqref{eq-ideal-pullback-torsion} with the locally free sheaf
$\Omega^1_{\tilde{X}}$ therefore gives
\[
0\longrightarrow \jmath_*(\mathcal O_E\oplus\mathcal O_E(-3))
\longrightarrow\Omega^1_{\tilde{X}}\otimes\pi^*F
\longrightarrow\Omega^1_{\tilde{X}}(-E)\longrightarrow0.
\]
By \eqref{counter-eq-1-1} with $W=\mathcal O_X$ and $(p,q)=(1,0)$,
$H^0(\tilde{X},\Omega^1_{\tilde{X}})\cong H^0(\mathbb P^2,\Omega^1_{\mathbb P^2})=0$,
and hence also
$H^0(\tilde{X},\Omega^1_{\tilde{X}}(-E))=0$.  Taking global sections yields
\[
H^0(\tilde{X},\Omega^1_{\tilde{X}}\otimes\pi^*F)
\cong H^0(E,\mathcal O_E\oplus\mathcal O_E(-3))\cong\mathbb C.
\]
On the other hand, $\Omega_X^1\otimes F$ is a subsheaf of $\Omega_X^1$,
whose space of global sections is zero.  This proves
\eqref{counter-eq-1-2}.

The blow-up formula in \cite[Theorem~1.2]{RYY20} cannot be extended to singular varieties
either.  For any integers $0<k<n$, D.~Arapura--D.~B.~Jaffe \cite{AJ} constructed a
projective Cohen--Macaulay variety $X$ of dimension $n$, an ample line bundle
$L$ on $X$, and a singular locus of codimension $k$ such that
$H^k(X,L^{-1})\neq0$.  If the formula held for blow-ups along arbitrary
smooth centers, then
$\dim_{\mathbb C}H^k(\tilde{X},\pi^*L^{-1})$ could not be smaller than
$\dim_{\mathbb C}H^k(X,L^{-1})$.

Choose a resolution given by a finite sequence of blow-ups
\[
\widehat X=X_n\xrightarrow{\pi_{n-1}}X_{n-1}\longrightarrow\cdots
\longrightarrow X_1\xrightarrow{\pi_0}X_0=X,
\]
where each $\pi_i:X_{i+1}\longrightarrow X_i$ has a smooth center and $\widehat X$
is a smooth projective variety of dimension $n$.  For the composite
$f=\pi_0\circ\pi_1\circ\cdots\circ\pi_{n-1}$, the assumed formula would give
\[
\dim_{\mathbb C}H^k(\widehat X,f^*L^{-1})
\geq\dim_{\mathbb C}H^k(X,L^{-1}).
\]
On the other hand, $f^*L$ is big and nef: indeed, $L$ is ample and $f$ is
birational.  Kawamata--Viehweg vanishing gives
$H^{n-k}(\widehat X,K_{\widehat X}\otimes f^*L)=0$, while Serre duality implies
$H^k(\widehat X,f^*L^{-1})=0$.  This contradicts $H^k(X,L^{-1})\neq0$.

\subsection{Failure of the de Rham blow-up formula for singular spaces}

Here we consider ordinary singular cohomology with complex coefficients,
which agrees with de Rham cohomology in the smooth case.  Let $X$ be the
projective cone over a nonsingular projective curve $C\subseteq\mathbb P^n$
of genus $g>0$.  Its ordinary cohomology groups have dimensions
\cite[Example~2.2.2]{ACM}
\[
\bigl(\dim H^l(X,\mathbb C)\bigr)_{l=0}^4=(1,0,1,2g,1).
\]
Let $Z=\{p\}$ be the vertex and $\pi:\tilde{X}\longrightarrow X$
its blow-up.  Extending the smooth blow-up formula unchanged would give
\begin{equation}\label{counter-eq-5-1}
H^l(\tilde{X},\mathbb C)
\cong H^l(X,\mathbb C)\oplus H^{l-2}(Z,\mathbb C).
\end{equation}
However, $\tilde{X}\cong\mathbb P(\mathcal O_C\oplus\mathcal O_C(1))$,
so the projective-bundle formula gives
\begin{equation}\label{counter-eq-5-2}
H^l(\tilde{X},\mathbb C)
\cong H^l(C,\mathbb C)\oplus H^{l-2}(C,\mathbb C).
\end{equation}
Consequently,
\[
H^1(\tilde{X},\mathbb C)\cong\mathbb C^{2g}\neq0
=H^1(X,\mathbb C)\oplus H^{-1}(Z,\mathbb C),
\]
which disproves \eqref{counter-eq-5-1} in this singular setting.

For this particular resolution, replacing ordinary cohomology on $X$ by
intersection cohomology does give a decomposition
\begin{equation*}
H^l(\tilde{X},\mathbb C)
\cong IH^l(X,\mathbb C)\oplus H^{l-2}(Z,\mathbb C).
\end{equation*}
Indeed, the intersection cohomology dimensions are \cite[Example~2.2.2]{ACM}
\begin{equation*}
\bigl(\dim IH^l(X,\mathbb C)\bigr)_{l=0}^4=(1,2g,1,2g,1),
\end{equation*}
whereas \eqref{counter-eq-5-2} gives $(1,2g,2,2g,1)$ for $\tilde{X}$.
The only additional summand is $H^0(Z,\mathbb C)$ in degree two.

\subsection{Failure of a naive K\"unneth formula for Bott--Chern cohomology}

We conclude with the product behavior of Bott--Chern cohomology.
For compact complex manifolds $X$ and $Y$, Chiose--R.~R\u{a}sdeaconu
\cite[Theorem~B]{CR} established a weak K\"unneth formula for Bott--Chern cohomology.
Stelzig \cite[Theorem 1.35 and Corollary 1.36]{Se25PH} gives general
K\"unneth formulae involving correction terms.  His
\cite[Example 1.34]{Se25PH} shows, at the level of abstract double
complexes, that the natural Bott--Chern tensor-product map can be neither
injective nor surjective.  Here we give an explicit computation on the
product of two Inoue surfaces showing the failure of the naive formula
  \begin{equation}\label{counter-eq-7-1}
H_{\mathrm{BC}}^{p,q}(X \times Y) \cong
\bigoplus_{\substack{i+j=p \\ k+l=q}}
H_{\mathrm{BC}}^{i,k}(X) \otimes H_{\mathrm{BC}}^{j,l}(Y).
\end{equation}
If one factor satisfies the $\partial\bar\partial$-lemma, however,
\eqref{counter-eq-7-1} holds via exterior products
\cite[Corollary 1.40]{Se25PH}.

Let $\mathcal A_X$ and $\mathcal A_Y$ be the double complexes of smooth
forms on $X$ and $Y$, respectively.  By \cite{GH}, pullback by the two
projections induces an $E_1$-isomorphism
\[
\mathcal A_X\otimes\mathcal A_Y\longrightarrow\mathcal A_{X\times Y}.
\]
By \cite[Lemma 2.4]{Se21},
\[
H_{\mathrm{BC}}^{p,q}(\mathcal A_X\otimes\mathcal A_Y)
\cong H_{\mathrm{BC}}^{p,q}(\mathcal A_{X\times Y}).
\]
Thus \eqref{counter-eq-7-1} is equivalent to
\begin{equation*}
H_{\mathrm{BC}}^{p,q}(\mathcal{A}_X \otimes\mathcal{A}_Y) \cong
\bigoplus_{\substack{i+j=p \\ k+l=q}}
H_{\mathrm{BC}}^{i,k}(\mathcal A_X) \otimes
H_{\mathrm{BC}}^{j,l}(\mathcal A_Y).
\end{equation*}
The counterexample comes from compact complex surfaces $S$ of Inoue type
\cite[\S~3]{ADT}.  They do not satisfy the $\partial\bar\partial$-lemma, and
\[
\dim_{\mathbb C}H^{p,q}_{\mathrm{BC}}(S)=1
\quad\text{if }(p,q)=(0,0),(1,1),(2,1),(1,2),(2,2),
\]
whereas $\dim_{\mathbb C}H^{p,q}_{\mathrm{BC}}(S)=0$ for all other
$(p,q)$.
Using the description of surface double complexes in
\cite[Remark 19]{S1}, together with the cohomology above, we obtain
\[
\mathcal A_S\simeq_1 D_0\oplus A\oplus B\oplus D_2,
\qquad
D_0=S_0^{0,0},\quad A=S_1^{0,0},\quad
B=S_3^{2,2},\quad D_2=S_4^{2,2}.
\]
Here $\simeq_1$ denotes $E_1$-equivalence, and we use the notation
$S_d^{p,q}$ of \cite[\S~2]{S1} for odd zigzags.  The summands $D_0$ and $D_2$
are dots in bidegrees $(0,0)$ and $(2,2)$, respectively.  The zigzag $A$ has nonzero components in bidegrees
$(0,1),(1,0),(1,1)$, with both arrows pointing to $(1,1)$, while $B$ has
nonzero components in $(1,1),(1,2),(2,1)$, with both arrows pointing away
from $(1,1)$.

The tensor-product rule of \cite[Proposition 16]{S1} gives
\[
A\otimes A\simeq_1 S_2^{0,0},\qquad
A\otimes B\simeq_1 S_4^{2,2},\qquad
B\otimes B\simeq_1 S_6^{4,4}.
\]
In bidegree $(1,2)$, the only contributions to the Bott--Chern cohomology
of $(D_0\oplus A\oplus B\oplus D_2)^{\otimes2}$ come from
$D_0\otimes B$, $B\otimes D_0$, and $A\otimes A$.
The first two each contribute one dimension.  The zigzag $S_2^{0,0}$
has sources in $(0,2),(1,1),(2,0)$ and sinks in $(1,2),(2,1)$,
so it also contributes one dimension in bidegree $(1,2)$.
Thus
\[
\dim_{\mathbb C}H_{\mathrm{BC}}^{1,2}(S\times S)=3.
\]
On the other hand, the only nonzero summands on the right-hand side of
\eqref{counter-eq-7-1} in this bidegree are
\[
H_{\mathrm{BC}}^{0,0}(S)\otimes H_{\mathrm{BC}}^{1,2}(S)
\quad\text{and}\quad
H_{\mathrm{BC}}^{1,2}(S)\otimes H_{\mathrm{BC}}^{0,0}(S),
\]
whose dimensions sum to $2$.  Hence \eqref{counter-eq-7-1} fails for
$X=Y=S$ and $(p,q)=(1,2)$.

\appendix
\section{A homological-algebra counterexample}
\label{appendix-homological-counterexample}

Consider the following commutative diagram of complex vector spaces with
exact rows:
\[
\begin{adjustbox}{max width=\linewidth}
\small$
\begin{tikzcd}[ampersand replacement=\&, column sep=1.5cm, row sep=2.4em]
0 \arrow[r] \arrow[d, "i_1 = 0"'] \& A_2 = \mathbb{C} \arrow[r, "f_2 = \operatorname{id}"] \arrow[d, "i_2 = \operatorname{id}"']
\& A_3 = \mathbb{C} \arrow[r, "f_3 = 0"] \arrow[d, "i_3{(x)=(x,0)}"']
\& A_4 = \mathbb{C} \arrow[r, "f_4 = \operatorname{id}"] \arrow[d, "i_4 = \operatorname{id}"']
\& A_5 = \mathbb{C} \arrow[r] \arrow[d, "i_5 = 0"] \& 0 \\
0 \arrow[r] \& B_2 = \mathbb{C} \arrow[r, "g_2{(x)=(x,0)}"]
\& B_3 = \mathbb{C}^2 \arrow[r, "g_3{(x,y)=y}"]
\& B_4 = \mathbb{C} \arrow[r, "g_4 = 0"]
\& B_5 = 0 \arrow[r] \& 0.
\end{tikzcd}
$
\end{adjustbox}
\]
This gives a counterexample to \cite[Proposition 5.1]{RYY20} and
\cite[Proposition 3.3]{YY20} as stated.  Nevertheless, the arguments in both papers
remain valid after a suitable revision: it suffices to assume additionally
that $i_5$ is injective.  In \cite[proof of Proposition~3.3]{YY20}, this injectivity follows from the
proposition below.

\begin{prop}
Suppose that $f:Y\longrightarrow X$ is a surjective holomorphic map between
compact complex manifolds of the same dimension.  Then
\[
f^*: \mathbb H^k\bigl(X,\mathcal B_X^{p,q}(\mathbb C)\bigr)
\longrightarrow
\mathbb H^k\bigl(Y,\mathcal B_Y^{p,q}(\mathbb C)\bigr)
\]
is injective for all $p,q\in\mathbb N$ and $k\in\mathbb Z$.
\end{prop}

\begin{proof}
It suffices to prove that
$f^*:\mathbb H^k(X,\mathcal L_X^{p,q})
\longrightarrow\mathbb H^k(Y,\mathcal L_Y^{p,q})$ is injective.  In the
definition of $\mathcal L_X^{p,q}$, replace each component in bidegree
$(s,t)$ by $\mathcal{K}_X^{s,t}$, the sheaf of $(s,t)$-currents on $X$,
and replace the differentials by their duals.  Denote the resulting complex
by $\mathcal K_X^{p,q}$.  By \cite[Lemma~2.1]{W}, the following diagram commutes up to
multiplication by $\deg f$:
\[
\begin{adjustbox}{max width=\linewidth}
\small$
\begin{tikzcd}[ampersand replacement=\&, column sep=1.8em, row sep=1.8em]
\mathcal L_X^{p,q} \ar[r, "\lambda_X"] \ar[d, "f^*"] \& \mathcal K_X^{p,q} \\
f_*\mathcal L_Y^{p,q} \ar[r, "\lambda_Y"] \& f_*\mathcal K_Y^{p,q}, \ar[u, "f_*"]
\end{tikzcd}
$
\end{adjustbox}
\]
where $\lambda_X$ and $\lambda_Y$ are the natural inclusions.
Applying hypercohomology to the diagram and retaining the same notation for
the induced maps, we obtain
\[
\deg(f)\lambda_X=f_*\lambda_Yf^*:
\mathbb H^k(X,\mathcal L_X^{p,q})
\longrightarrow\mathbb H^k(Y,\mathcal K_Y^{p,q}).
\]
Since $\lambda_X$ is a quasi-isomorphism, its induced map on hypercohomology is an
isomorphism; hence $f^*$ is injective.
\end{proof}


\begin{thebibliography}{BHPV04}

\bibitem[Al17]{Al17}
L. Alessandrini,
\textit{Proper modifications of generalized $p$-K\"ahler manifolds},
J. Geom. Anal. 27 (2017), no.~2, 947--967.

\bibitem[ADT16]{ADT}
D. Angella, G. Dloussky, A. Tomassini,
\textit{On Bott--Chern cohomology of compact complex surfaces},
Ann. Mat. Pura Appl. 195 (2016), 199--217.

\bibitem[ASTT20]{ASTT20}
D. Angella, T. Suwa, N. Tardini, A. Tomassini,
\textit{Note on Dolbeault cohomology and Hodge structures up to
bimeromorphisms},
Complex Manifolds 7 (2020), no.~1, 194--214.

\bibitem[AT13]{AT13}
D. Angella, A. Tomassini,
\textit{On the $\partial\bar\partial$-lemma and Bott--Chern cohomology},
Invent. Math. 192 (2013), no.~1, 71--81.

\bibitem[AJ89]{AJ}
D. Arapura, D. B. Jaffe,
\textit{On Kodaira vanishing for singular varieties},
Proc. Amer. Math. Soc. 105 (1989), 911--916.

\bibitem[AHS78]{AHS78}
M. F. Atiyah, N. J. Hitchin, I. M. Singer,
\textit{Self-duality in four-dimensional Riemannian geometry},
Proc. Roy. Soc. London Ser. A 362 (1978), 425--461.

\bibitem[BHPV04]{BHPV04}
W. P. Barth, K. Hulek, C. A. M. Peters, A. Van de Ven,
\textit{Compact complex surfaces}, second enlarged ed.,
Ergebnisse der Mathematik und ihrer Grenzgebiete, vol.~4,
Springer, Berlin, 2004.

\bibitem[BM97]{BM97}
E. Bierstone, P. D. Milman,
\textit{Canonical desingularization in characteristic zero by blowing up the
maximum strata of a local invariant},
Invent. Math. 128 (1997), no.~2, 207--302.

\bibitem[Bn83]{Bin83}
J. Bingener,
\textit{On deformations of K\"ahler spaces. I},
Math. Z. 182 (1983), no. 4, 505--535.

\bibitem[BT82]{BT82}
R. Bott, L. W. Tu,
\textit{Differential forms in algebraic topology},
Graduate Texts in Mathematics, vol.~82,
Springer-Verlag, New York, 1982.

%
\bibitem[CY22]{CSY}
Youming Chen, Song Yang,
\textit{On blow-up formula of integral Bott--Chern cohomology},
Ann. Global Anal. Geom. 61 (2022), 57--67.

\bibitem[Ch14]{Ch14}
I. Chiose,
\textit{Obstructions to the existence of K\"ahler structures on compact
complex manifolds},
Proc. Amer. Math. Soc. 142 (2014), no.~10, 3561--3568.



\bibitem[CR22]{CR}
I. Chiose, R. R\u{a}sdeaconu,
\textit{Remarks on astheno-K\"ahler manifolds, Bott--Chern and Aeppli
cohomology groups},
Ann. Global Anal. Geom. 63 (2023), no. 3, Paper No. 24, 23 pp.

\bibitem[dC11]{D}
M. A. A. de Cataldo,
\textit{Lectures on perverse sheaves and decomposition theorem},
University of Michigan, Ann Arbor, October 2011.

\bibitem[dCM09]{ACM}
M. A. A. de Cataldo, L. Migliorini,
\textit{The decomposition theorem, perverse sheaves and the topology of
algebraic maps},
Bull. Amer. Math. Soc. (N.S.) 46 (2009), 535--633.

\bibitem[DGMS75]{DGMS75}
P. Deligne, P. Griffiths, J. Morgan, D. Sullivan,
\textit{Real homotopy theory of K\"ahler manifolds},
Invent. Math. 29 (1975), no.~3, 245--274.

\bibitem[Dm12]{De}
J.-P. Demailly,
\textit{Complex analytic and differential geometry},
\href{https://www-fourier.ujf-grenoble.fr/~demailly/manuscripts/agbook.pdf}
{online book}, 2012.

\bibitem[Fj00]{Fj00}
A. Fujiki,
\textit{Compact self-dual manifolds with torus actions},
J. Differential Geom. 55 (2000), no.~2, 229--324.

\bibitem[GH94]{GH}
P. Griffiths, J. Harris,
\textit{Principles of algebraic geometry}, reprint of the 1978 original,
Wiley Classics Library, John Wiley \& Sons, Inc., New York, 1994.

\bibitem[Gr85]{G}
M. Gros,
\textit{Classes de Chern et classes de cycles en cohomologie de Hodge--Witt
logarithmique},
Bull. Soc. Math. France M\'emoire 21 (1985), 1--87.

\bibitem[GNA02]{GNA}
F. Guill\'en, V. Navarro Aznar,
\textit{Un crit\`ere d'extension des foncteurs d\'efinis sur les sch\'emas
lisses},
Publ. Math. Inst. Hautes \`Etudes Sci. 95 (2002), 1--91.

\bibitem[GR65]{Gu}
R. C. Gunning, H. Rossi,
\textit{Analytic functions of several complex variables},
Prentice-Hall, Englewood Cliffs, NJ, 1965.

\bibitem[GZ26]{GZ26}
Yuqin Guo, Fangyang Zheng,
\textit{Streets--Tian conjecture on several special types of Hermitian
manifolds},
Ann. Mat. Pura Appl. (4) 205 (2026), no.~1, 101--118.

\bibitem[Ha02]{Ha02}
A. Hatcher,
\textit{Algebraic topology},
Cambridge University Press, Cambridge, 2002.

\bibitem[Hn08]{Hn08}
N. Honda,
\textit{Projective models of the twistor spaces of Joyce metrics},
preprint, \href{https://arxiv.org/abs/0805.0046}{arXiv:0805.0046}.

\bibitem[Hn15]{Hn15}
N. Honda,
\textit{Geometry of some twistor spaces of algebraic dimension one},
Complex Manifolds 2 (2015), 105--130.

\bibitem[Hr76]{Hr76}
E. Horikawa,
\textit{Deformations of holomorphic maps. III},
Math. Ann. 222 (1976), 275--282.

\bibitem[Jy95]{Jy95}
D. Joyce,
\textit{Explicit construction of self-dual $4$-manifolds},
Duke Math. J. 77 (1995), no.~3, 519--552.

\bibitem[KS90]{KS}
M. Kashiwara, P. Schapira,
\textit{Sheaves on manifolds},
Grundlehren der mathematischen Wissenschaften, vol.~292,
Springer-Verlag, Berlin, 1990.

\bibitem[Ki71]{King}
J. R. King,
\textit{The currents defined by analytic varieties},
Acta Math. 127 (1971), 185--220.

\bibitem[Kd66]{Kd66}
K. Kodaira,
\textit{Complex structures on $S^1\times S^3$},
Proc. Natl. Acad. Sci. USA 55 (1966), no.~2, 240--243.

\bibitem[La99]{La99}
A. Lamari,
\textit{Courants k\"ahl\'eriens et surfaces compactes},
Ann. Inst. Fourier (Grenoble) 49 (1999), no.~1, 263--285.

\bibitem[LN26]{LN26}
Tian-Jun Li, Shengzhen Ning,
\textit{The spaces of K\"ahler and holomorphically tamed symplectic forms on
closed $4$-manifolds},
preprint, \href{https://arxiv.org/abs/2607.18778}{arXiv:2607.18778}.

\bibitem[Mn19]{Mn19MV}
Lingxu Meng,
\textit{Mayer--Vietoris systems and their applications},
preprint, \href{https://arxiv.org/abs/1811.10500v3}{arXiv:1811.10500v3}, 2019.

\bibitem[Mn20a]{M}
Lingxu Meng,
\textit{Leray--Hirsch theorem and blow-up formula for Dolbeault cohomology},
Ann. Mat. Pura Appl. (4) 199 (2020), 1997--2014.

\bibitem[Mn20b]{Mn20BC}
Lingxu Meng,
\textit{Blow-up formulae for twisted cohomologies with supports},
preprint, \href{https://arxiv.org/abs/2010.03102}{arXiv:2010.03102}, 2020.

\bibitem[Mn21]{Mn21}
Lingxu Meng,
\textit{The heredity and bimeromorphic invariance of the
$\partial\bar\partial$-lemma property},
C. R. Math. 359 (2021), no.~6, 645--650.

\bibitem[OVV21]{OVV21}
L. Ornea, M. Verbitsky, V. Vuletescu,
\textit{Classification of non-K\"ahler surfaces and locally conformally
K\"ahler geometry},
Russian Math. Surveys 76 (2021), no.~2, 261--289.

%
\bibitem[Pr66]{Pr66}
A. N. Parshin,
\textit{A generalization of the Jacobian variety},
Izv. Akad. Nauk SSSR Ser. Mat. 30 (1966), no.~1, 175--182 (in Russian).

\bibitem[RYY19]{RYY19}
Sheng Rao, Song Yang, Xiangdong Yang,
\textit{Dolbeault cohomologies of blowing up complex manifolds},
J. Math. Pures Appl. 130 (2019), 68--92.

\bibitem[RYY20]{RYY20}
Sheng Rao, Song Yang, Xiangdong Yang,
\textit{Dolbeault cohomologies of blowing up complex manifolds II:
bundle-valued case},
J. Math. Pures Appl. 133 (2020), 1--38.

%
\bibitem[RZ24]{RZ24}
Sheng Rao, Yongpan Zou,
\textit{$\partial\bar\partial$-lemma and double complex},
Commun. Math. Stat. (2024), 1--42,
\href{https://doi.org/10.1007/s40304-024-00400-x}
{doi:10.1007/s40304-024-00400-x}.

%
\bibitem[Sc07]{Sch}
M. Schweitzer,
\textit{Autour de la cohomologie de Bott--Chern},
preprint, \href{https://arxiv.org/abs/0709.3528}{arXiv:0709.3528}.

\bibitem[Se55]{Serre}
J.-P. Serre,
\textit{Un th\'eor\`eme de dualit\'e},
Comment. Math. Helv. 29 (1955), 9--26.

\bibitem[St18]{Se18}
J. Stelzig,
\textit{Double complexes and Hodge structures as vector bundles},
Ph.D. thesis, University of M\"unster, 2018.

\bibitem[St21a]{Se21}
J. Stelzig,
\textit{The double complex of a blow-up},
Int. Math. Res. Not. IMRN 2021, no.~14, 10731--10744.

\bibitem[St21b]{S1}
J. Stelzig,
\textit{On the structure of double complexes},
J. London Math. Soc. (2) 104 (2021), 956--988.

\bibitem[St25a]{Se25}
J. Stelzig,
\textit{Some remarks on the Schweitzer complex},
Ann. Inst. Fourier (Grenoble) 75 (2025), no.~1, 35--47.

\bibitem[St25b]{Se25PH}
J. Stelzig,
\textit{Pluripotential homotopy theory},
Adv. Math. 460 (2025), Paper No.~110038, 61 pp.

\bibitem[ST10]{ST10}
J. Streets, Gang Tian,
\textit{A parabolic flow of pluriclosed metrics},
Int. Math. Res. Not. IMRN 2010, no.~16, 3101--3133.

\bibitem[Vi02]{Vi02}
C. Voisin,
\textit{Hodge theory and complex algebraic geometry. I},
Cambridge Studies in Advanced Mathematics, vol.~76,
Cambridge University Press, Cambridge, 2002.

\bibitem[We74]{W}
R. O. Wells,
\textit{Comparison of de Rham and Dolbeault cohomology for proper surjective
mappings},
Pacific J. Math. 53 (1974), 281--300.

\bibitem[Wc06]{Wu06}
Chun-Chun Wu,
\textit{On the geometry of superstrings with torsion},
Ph.D. thesis, Harvard University, 2006.

\bibitem[Wx23]{Wu}
Xiaojun Wu,
\textit{Intersection theory and Chern classes in Bott--Chern cohomology},
Ark. Mat. 61 (2023), 211--265.

\bibitem[YY20]{YY20}
Song Yang, Xiangdong Yang,
\textit{Bott--Chern blow-up formula and bimeromorphic invariance of the
$\partial\bar\partial$-lemma for threefolds},
Trans. Amer. Math. Soc. 373 (2020), no.~12, 8885--8909.

\bibitem[YY23]{SY}
Song Yang, Xiangdong Yang,
\textit{Bott--Chern hypercohomology and bimeromorphic invariants},
Complex Manifolds 10 (2023), 1--38.

\end{thebibliography}
\end{document}